\documentclass[12pt]{amsart}
\theoremstyle{plain}
\usepackage{fullpage, url, amssymb, bm}
\usepackage{enumitem, graphicx}
\usepackage[all]{xy} 
\usepackage{xcolor}
\usepackage{accents}
\usepackage{amsmath}
\usepackage{amsthm}
\DeclareFontEncoding{OT2}{}{} 

\font\elevensc=cmcsc10 scaled \magstephalf

\newcommand{\hra}{\hookrightarrow}

\newcommand{\hor}[1]{\smash
         {\mathop{{\lgrghtar}}\limits^{\lower2pt\hbox{$\scriptstyle{#1}$}}}}
\newcommand{\horr}[1]{\smash
{\mathop{{\lglgrghtar}}\limits^{\lower2pt\hbox{$\scriptstyle{#1}$}}}}
\newcommand{\hr}[1]{\smash
{\mathop{{\to}}\limits^{\lower2pt\hbox{$\scriptstyle{#1}$}}}}

\newcommand{\lgrghtar}{{\ha2{\relbar\joinrel\rightarrow}\ha2}}
\newcommand{\lglgrghtar}{{\ha1{\relbar\joinrel\relbar\joinrel\rightarrow}\ha1}}

\newcommand{\nmnm}[1]{{\elevensc #1}}

\newcommand{\rsdp}{{\,\times\kern-3pt\lower-1pt%
\hbox{$\scriptscriptstyle|$\ha3}}}

\newcommand{\srjr}{\twoheadrightarrow}

\newcommand{\plim}[1]{\hbox to14pt{%
lim\kern-14pt\lower4.5pt\hbox{$\scriptstyle\longleftarrow$}%
\kern-8pt\lower8.5pt\hbox{$\scriptstyle{#1}$}}\ha{3}}

\newcommand{\ilim}[1]{\hbox to14pt{%
lim\kern-14pt\lower4.5pt\hbox{$\scriptstyle\longrightarrow$}%
\kern-8pt\lower8.5pt\hbox{$\scriptstyle{#1}$}}\ha{3}}

\newcommand{\plm}[1]{\hbox to14pt{%
lim\kern-14pt\lower5pt\hbox{$\scriptstyle\longleftarrow$}%
\kern-8pt\lower8.5pt\hbox{$\scriptstyle{#1}$}}\ha{3}}

\newcommand{\ilm}[1]{\hboxto14pt{%
lim\kern-14pt\lower5pt\hbox{$\scriptstyle\longrightarrow$}%
\kern-8pt\lower8.5pt\hbox{$\scriptstyle{#1}$}}\ha{3}}

\newcommand{\clC}{{\mathcal C}}
\newcommand{\clD}{{\mathcal D}}

\newcommand{\clO}{{\mathcal O}}
\newcommand{\clP}{{\mathcal P}}

\newcommand{\clT}{{\mathcal T}}
\newcommand{\clU}{{\mathcal U}}

\newcommand{\lvF}{{\mathbb F}}

\newcommand{\lvN}{{\mathbb N}}

\newcommand{\lvP}{{\mathbb P}}
\newcommand{\lvQ}{{\mathbb Q}}

\newcommand{\lvZ}{{\mathbb Z}}

\let\eu=\mathfrak

\newcommand{\eum}{{\eu m}}

\newcommand{\euX}{{\eu X}}

\DeclareMathOperator{\chr}{char}
\DeclareMathOperator{\codim}{codim}
\DeclareMathOperator{\Div}{Div}
\DeclareMathOperator{\Frob}{Frob}

\DeclareMathOperator{\Proj}{Proj}

\DeclareMathOperator{\Spec}{Spec}

\newcommand{\Aut}[2]{{\rm Aut}{#1}(#2)}

\newcommand{\Bstar}{{\rm{\sf(}\lower2pt\hbox{\LARGE$*$}{\sf)}}}

\newcommand{\CCT}{{\rm{\defi{(CCT)}}}}

\newcommand{\defi}[1]{\textsf{#1}}

\newcommand{\hb}[1]{\hbox to-#1pt{}}
\newcommand{\ha}[1]{\hbox to#1pt{}}
\newcommand{\hmm}{\hb1}

\newcommand{\itm}[2]{%
\begin{itemize}[leftmargin=#1pt]
#2
\end{itemize}
}

\newcommand{\jspl}{$\jmath$\ha1-\ha{.5}split}

\newcommand{\Kv}{{K\!v}}  
  
\newcommand{\kv}{{k_v}}

\newcommand{\Lv}{{L\hmm v}}
\newcommand{\Lw}{{L\hmm w}}
\newcommand{\lps}[1]{{(\hb1(#1)\hb1)}}

\newcommand{\Mv}{M\!v}
\newcommand{\mclU}{{\scriptscriptstyle\clU}}
\newcommand{\md}{\ha1|\ha1}

\newcommand{\Nv}{N\hb{1.5}v}
\newcommand{\Nw}{N\hb{1.5}w}
\newcommand{\nix}{{\phantom{|}}}

\newcommand{\oli}{\overline}

\newcommand{\Srj}{({\sf Srj})}
\newcommand{\Sgm}[1]{\Sigma_{#1}}

\newcommand{\td}{{\rm td}}
\newcommand{\tl}{\tilde}

\newcommand{\tlv}{{\tilde v}}

\newcommand{\UU}{{\scriptscriptstyle\clU}}
\newcommand{\Ust}[1]{{}^*\hb{1.5}{#1}_\UU}
\newcommand{\Uj}[1]{{#1}_\UU}
\newcommand{\Us}[1]{{#1}_\UU}

\newcommand{\Val}[2]{{\rm Val}_{#1}(#2)}

\newcommand{\vrl}{valuation-regular-like}

\newcommand{\vid}{{\not{\hb3\lower-1pt\hbox{$\scriptstyle\bigcirc$}}}}

\newtheorem{theorem}{Theorem}[section]
\newtheorem*{theorem*}{Theorem}

\newtheorem{corollary}[theorem]{Corollary}
\newtheorem{keylemma}[theorem]{Key Lemma}
\newtheorem{lemma}[theorem]{Lemma}
\newtheorem{proposition}[theorem]{Proposition}

\theoremstyle{definition}

\newtheorem{definition}[theorem]{Definition}
\newtheorem{definition/remark}[theorem]{Definition/Remark}
\newtheorem{example}[theorem]{Example}

\newtheorem{example/fact}[theorem]{Example/Fact}
\newtheorem{fact}[theorem]{Fact}

\newtheorem{fact/definition}[theorem]{Fact/Definition}

\newtheorem{notations/remarks}[theorem]{Notations/Remarks}

\newtheorem{remark/definition}[theorem]{Remark/Definition}
\newtheorem{remarks/examples}[theorem]{Remarks/Examples}

\begin{document}

\title[\textbf{On a Conjecture of Colliot-Th\'el\`ene}]
     {\textbf{On a Conjecture of Colliot-Th\'el\`ene}}

\author{Florian Pop}

\address{Department of Mathematics, University of Pennsylvania
        \vskip0pt
        DRL, 209 S 33rd Street, Phila\-delphia, PA 19104, USA}
\email{pop@math.upenn.edu}
\urladdr{https://www.math.upenn.edu/\~{}pop/}

\begin{abstract} 
The aim of this short note is to extend results by \nmnm{Denef}
and \nmnm{Loughran, Skorobogatov, Smeets} concerning  
refinements of a conjecture of \nmnm{Colliot-Th\'el\`ene}. The
problem is about giving necessary and sufficient conditions for
morphisms of varieties to be surjective on local points for almost
all localizations. 
\end{abstract}


\vskip4pt
\thanks{\textit{\textbf{NOT funded by the NSF}\/}: The 
ANT research directors (Douglas, Hodges, Libgober 
and Pollington), in their higher understanding of this 
matter, overruled the expert reviewers and panelists 
recommendation and decided not to fund this research.}
\date{Variant of August 31, 2019.}

\maketitle

\section{Introduction/Motivation}
The aim of this note is to shed new light on a 
conjecture by Colliot-Th\'el\`ene, cf.~\cite{CT}, 
concerning the image of local rational points under 
dominant morphisms of varieties over global
fields. The precise context is as follows: 
\vskip2pt
\itm{20}{
\item[-] Let $k$ be a {\it global field,\/} $\lvP(k)$ be the places 
of $k$, and $\kv$ be the completion of $k$ at $v\in \lvP(k)$. 
\vskip2pt
\item[-] Let $f:X \to Y$ be a morphism of $k$-varieties. 
}
For every $v\in\lvP(k)$, the $k$-morphism
$f$ gives rise to a canonical map $f^\kv:X(\kv)\to Y(\kv)$.
There are obvious examples showing that, in general, $f^\kv$ 
is not surjective, e.g.\ $f:\lvP^1_\lvQ\to\lvP^1_\lvQ$ of
degree two. Therefore, for $f:X\to Y$ as above, it is natural to
consider the basic property:
\vskip7pt
\centerline{\Srj \ha{60} {\it $f^\kv:X(\kv)\to Y(\kv)$ 
is surjective for almost all $v\in\lvP(k)$.\/}\ha{80}}
\vskip7pt
\noindent
and to ask the following fundamental:
\vskip7pt
\centerline{{\bf Question:}  {\it Give {necessary 
and sufficient} conditions for $\,f:X\to Y$ to have 
property\/} \Srj.} 
\vskip7pt
\noindent
This problem was considered in a systematic way 
by \nmnm{Colliot-Th\'el\`ene}~\cite{CT}, under 
the following restrictive but to some extent natural 
hypothesis:
\begin{displaymath}
\hbox{\defi{$(*)$}\ha{27}} 
\left. \begin{array}{c}
\hbox{\it $k\,$ is a \textit{\textbf{number field}}, 
$X$, $Y$ are projective smooth integral 
$k$-varieties, and\ha{13}\/}\\
\hbox{\it $f:X\!\to Y$ is a dominant morphism with 
geometrically integral generic fiber.\ha{10}}
\end{array} \right.
\end{displaymath} 
In particular, if $L:=k(Y)$ is the function field of 
$Y\!$, the generic fiber $X_L$ of $f$ can be
viewed as an $L$-variety. In this notation, 
for morphisms $f:X\to Y$ satisfying~$(*)$, 
\nmnm{Colliot-Th\'el\`ene} considered the 
hypothesis~\defi{(CT)} and made the 
conjecture~\CCT\ below:
\begin{displaymath}
\hbox{\hb{40}\defi{(CT)}} \ha{30} 
\left. \begin{array}{c}
\hbox{\it For each discrete valuation $k$-ring 
$\,R\subset L$, and its residue field $\,\kappa_R$,}\\
\hbox{\it there is a \textbf{regular} flat 
$\,R\ha1$-model $\,\euX_R$ of $\,X_L$ 
whose special fiber $\,\euX_{\kappa_R}$}\\
\hbox{\it has an irreducible component 
$\euX_\alpha$ which is $\kappa_R$-geometrically 
integral.} 
\end{array} \right.
\end{displaymath}
\noindent
\textbf{Conjecture of Colliot-Th\'el\`ene \CCT.}
{\it Let $f:X\to Y$ be a dominant morphism~of~proper 
smooth geometrically integral varieties over a number
field~$k$, and suppose that hypotheses~$(*)$ 
and~{\rm\defi{(CT)}} are satisfied. Then $f:X\to Y$ 
has the property\/} \Srj.
\vskip7pt
In a recent paper, \nmnm{Denef}~\cite{Df2} proved 
a stronger form of the conjecture \CCT, by replacing the 
hypothesis~\defi{(CT)} by the weaker hypothesis~\defi{(D)} 
below. In order to explain \nmnm{Denef}'s result,  
recall the following terminology: Let $f:X\!\to Y$
be a morphism satisfying hypothesis~$(*)$. A 
\defi{(smooth) modification} of $f$ is any morphism $f':X'\to Y'$
satisfying hypothesis~$(*)$ such that there exist 
modifications (i.e., birational morphisms) $p:X'\to X$,
$q:Y'\to Y$ satisfying $q\circ f'=f\circ p$. Given 
a smooth modification $f':X'\!\to Y'$ of $f$, for every Weil 
prime divisor $E'\subset Y'\!$, and the Weil 
prime divisors $D'$ of $X'$ above $E'\!$, consider: 
First, the multiplicity $e(D'|E')$ of $D'$ 
in ${f'}^*(E')\in\Div(X')$; second, the restriction
$f'_{D'}:D'\to E'$ of $f'$ to $D'\subset X'\!,$ 
which is a morphism of integral $k$-varieties. 
Finally, for $f:X\to Y$ satisfying~$(*)$, it turns out 
that the hypothesis \defi{(CT)} above implies that 
following obviously weaker hypothesis:
\begin{displaymath}
\hbox{\defi{(D)}} \ha{5} 
\left. \begin{array}{c}
\hbox{\it For every modification $f'$ and every Weil 
prime divisor $E'\subset Y'\!$, there is $D'$ above }\\
\hbox{\it $E'$ with $e(D'|E)=1$ and $f'_{D'}:D'\to E'$ 
having geometrically integral generic fiber.}
\end{array} \right.
\end{displaymath}
\begin{theorem*}
   [\nmnm{Denef}~\cite{Df2}, Main Theorem 1.2]
$\ha0$\vskip0pt
Let $f:X\to Y$ satisfy the hypotheses~$(*)$ 
and~{\rm\defi{(D)}}. Then $f$ has the property {\rm\Srj}.
\end{theorem*}
%
%
Finally recall the very recent results by 
\nmnm{Loughran--Skorobogatov--Smeets}~\cite{LSS}~which, 
for morphisms $f:X\to Y$ satisfying the
hypothesis~$(*)$ above, give \textit{\textbf{necessary 
and sufficient conditions}\/} such that $f:X\to Y$
has property \Srj. Namely, following~\cite{LSS}, 
in the notation introduced above, let $f':X'\to Y'$ 
be a smooth modification of $f:X\to Y\!$. For a 
Weil prime divisor $E'$ of $Y'\!$ and a Weil 
prime divisor $D'$ of $X'$ above $E'\!$, let
$k(D')\md k(E')$ be the function field extension 
defined by the dominant map $f'_{\!D'}:D'\to E'$. 
One says that $E'$ is \defi{pseudo-split} under 
$f':X'\to Y'\!$, if for every element of the absolute 
Galois group $\,\sigma\in G_{k(E')}$, there is 
some Weil prime divisor $D'$ of $X'$ above 
$E'$ satisfying: 
\[
e(D'|E')=1 \ \hbox{ and } \ k(D')
\otimes_{k(E')}\!\oli{k(E')} 
\hbox{ has a factor stabilized by } \sigma.
\]
Following \nmnm{Loughran--Skorobogatov--Smeets}
\cite{LSS}, consider the hypothesis:
\begin{displaymath}
\hbox{\defi{(LSS)}} 
\left. \begin{array}{c}
\hbox{\it For all smooth modifications $f'$ of $f\!$,
all Weil prime divisors $E'\!\subset Y'\!$ are pseudo-split.}
\end{array} \right.
\end{displaymath}
\noindent
Note that if $D'\!, E'$ satisfy hypothesis~\defi{(D)}, 
then $k(D')\md k(E')$ is a regular field extension, hence 
$k(D')\otimes_{k(E')}\oli{k(E')}$ is a field stabilized by 
all $\sigma\in G_{\kappa_{E'}}$\hb2 (and $E'$ is called 
\defi{split}). Hence~hypothesis~\defi{(D)} implies 
\defi{(LSS)}, leading to the following sharpening of 
\nmnm{Denef}'s result above:
\begin{theorem*}[\nmnm{Loughran--Skorobogatov--Smeets}
\cite{LSS},  Theorem 1.4] 
$\ha0$\vskip0pt
Let $f\!:\!X\!\to Y$ satisfy~$(*)$. Then $f$ satisfies 
hypothesis~{\rm\defi{(LSS)}} \ iff \ $f$ has property {\rm\Srj}.
\end{theorem*}
The aim of this note is to provide a different approach
to the basic problem and \CCT\ considered above, which 
among other things allows the following:
\vskip5pt
\itm{25}{
\item[$\bullet$] There are \textit{\textbf{no}\/} 
smoothness/properness/irreducibility hypotheses 
on the $k$-varieties $X,Y\!$.
\vskip2pt
\item[$\bullet$] Hypotheses \defi{(D)}, \defi{(LSS)} 
can be replaced by the \textit{\textbf{weaker 
hypotheses}\/} \Bstar, \Bstar$_{\Sgm k}$ below.
\vskip2pt
\item[$\bullet$] $k$ can be \textit{\textbf{more 
general,}\/} e.g.\ a finitely generated field
of characteristic zero, or finitely generated over  
a PAC field of characteristic zero.
\vskip2pt
\item[$\bullet$] Finally, in positive characteristic
$p>0$, we give sufficient condition for \Srj\ to hold, 
e.g.\ in the case $k$ is finitely generated, or finitely
generated over a PAC field.
}

\vskip5pt
In order to proceed, let us introduce/consider 
notation as follows: Let $N\md k$ be a function 
field over an arbitrary base field $k$. For a 
valuation $v$ of $N$, let $\clO_v,\eum_v$ be its 
valuation ring/ideal, $vN$ denote the valuation 
group, and $\Nv$ be the residue field of $v$. 
A valuation $v$ of $N$ is called a $k$-valuation
if $v$ is trivial on $k$, or equivalently, 
$k\subset\clO_v$. The space of $k$-valuations 
$\Val k N$ of $N|k$, called the Riemann--Zariski 
space of $N|k$, carries naturally the Zariski topology 
via the following geometric interpretation:
\vskip7pt
{\it Let $(Z_\alpha)_\alpha$ be any cofinal family 
of proper $k$-models of $Z_\alpha$ w.r.t.\ the 
domination relation. 
\vskip0pt
For $v\in\Val k N$, let $z_{\alpha,v}\in Z_\alpha$ 
be the center of $v$ on $Z_\alpha$. Then one has:\/} 
\[
\Val k N=\plim{\alpha}\,Z_\alpha,\quad
v=(z_{\alpha,v})_\alpha,\quad\clO_v=
\ilim\alpha\,\clO_{z_{\alpha,v}}, \ \
\eum_v=\ilim\alpha\,\eum_{z_{\alpha,v}}.
\] 

A $k$-valuation $v\in\Val k N$ is called a 
\defi{prime divisor} of $N\md k$ if there is a
normal model $Z$ of $N\md k$ and a Weil prime 
divisor $D$ of $Z$ with $\clO_v=\clO_{\eta_D}$, 
the local ring of the generic point $\eta_D\in Z$ 
of $D$. In particular, $vN=\lvZ$, and $\Nv=k(D)$ is the 
function field of the $k$-variety~$D$, thus satisfying 
$\td(\Nv|k)=\td(N|k)-1$. For $v\in \Val k N$ the 
following are equivalent:
\vskip2pt
\itm{30}{
\item[i)] $v$ is a prime divisor of $N|k$.
\vskip2pt
\item[ii)] In the above notation, the center 
$z_{\alpha,v}$ of $v$ on some $Z_\alpha$ has 
$\codim_{Z_\alpha}(z_{\alpha,v})=1$.
\vskip2pt
\item[iii)] $\td(\Nv\md k)=\td(N\md k)-1$.
}
\vskip5pt
Let $\clD(N|k)$ denote the {\it set of prime divisors of 
$N|k$ together with the trivial valuation.\/}
\vskip7pt
For extensions of function fields $M\md N$ over 
$k$, the restriction $\Val k M\to\Val k N$, 
$v\mapsto v|_N$ is surjective, and defines a 
surjective map $\clD(M|k)\to\clD(N|k)$. In 
particular, if $v\in\clD(M|k)$ and $w=v|_N$, then 
$e(v|w):=(vM:wN)$ is finite if either $v$ is trivial 
or $w$ is non-trivial, and there is a canonical 
$k$-embedding of the residue function fields
$\Lw:=\kappa(w)\hra\kappa(v)=:\Kv$.

%
%
We say that $w\in\clD(L|k)$ is \defi{pseudo-split} 
in $\clD(M|k)$, if for every $\sigma\in G_{\Lw}$, 
there is some $v\in\clD(M|k)$ satisfying: $w=v|_N$,  
$e(v|w)=1$ if $w$ is non-trivial, and 
$M\hb1v\otimes_{\Lw}\oli{L\ha2}\hb3w$ has a factor 
which is a field stabilized by $\sigma$. And we say
that $\clD(L|k)$ is pseudo-split in $\clD(M|k)$, 
if all $w\in\clD(L|k)$ are pseudo-split in $\clD(M|k)$.
\vskip2pt
The above notion of pseudo-splitness relates to the
one from \cite{LSS} mentioned above as follows: 
Let $f:X\to Y$ be a dominant morphism of 
projective smooth varieties over a~number field 
$k$, and setting $K=k(X)$, $L=k(Y)$, let $K\md L$ 
be the corresponding extension of function fields. Let  
$f_\alpha\!:\!X_\alpha\!\to\! Y_\alpha$, $\alpha\in I$ be
the (projective) system of all the smooth modifications 
of~$f$ satisfying the hypothesis~$(*)$. By Hironaka's 
Desingularization Theorem, $(X_\alpha)_\alpha$ 
and $(Y_\alpha)_\alpha$ are cofinal (w.r.t.\ the 
domination relation) in the system of all the 
proper models of $K|k$, respectively~$L|k$. 
Hence by mere definitions one has: 
\vskip5pt\noindent
{\bf Fact.}
{\it The hypothesis 
{\rm\defi{(LSS)}} implies that $\clD(L\md k\big)$ 
is pseudo-split in $\clD(K\md k)$.\/}
\vskip5pt
Finally, let $Z$ be an integral $k$-variety, and 
$N=k(Z)$ be its function field. A point $z\in Z$ is 
called \defi{valuation-regular-like (v.r.l.)}, if there 
exist $\tlv\in\Val k N$ and $v\in\clD(N|k)$ both 
having center $z\in Z$ such that $N\hb1\tlv=\kappa(z)$,
$\Nv|\kappa(z)$ is a regular field extension, and 
$v(u)=1$ for all $u\in\eum_z\backslash\eum_z^2$. 
Notice that {\it regular points $z\in Z$ are~v.r.l.\/}: 
Indeed, if $(t_1,\dots,t_d)$ is a system of regular 
parameters of $\clO_z$, the canonical $k$-embedding  
$K\hra \kappa(z)\lps{t_1}\dots\lps{t_d}$ defines
a valuation $\tlv\in\Val k K$ with $K\hb1\tlv=\kappa(z)$.
Further, the \defi{degree valuation} $v$ defined
by $(\eum_z^i)^\nix_i$ has as residue field the rational 
function field $\Kv=\kappa(z)(t_i/t_d)_{i<d}$ and satisfies 
$v(u)=1$ for all $u\in\eum_z\backslash\eum_z^2$. 
We say that $Z$ is valuation-regular-like, 
if all $z\in Z$ are v.r.l.\ha3points. Note that regular 
$k$-varieties are valuation-regular-like, but the 
converse does not hold: Indeed, rational double points 
and rational cusps of curves are v.r.l.\ points, but not 
regular points.
\vskip2pt
This being said, a first result extending/generalizing
and shedding new light on the afore mentioned 
\cite{Df2}, Main Theorem 1.2, and \cite{LSS},
Theorem 1.4, is as follows: 
\begin{theorem}
\label{thm0}
Let $K\md L$ be an extension of function fields 
over a number field $k$ defined by a~dominant 
morphism $f:X\to Y$ of proper 
valuation-regular-like $k$-varieties. Then $f$ has 
property {\rm\Srj} \ iff \ $\clD(L|k)$ is pseudo-split 
in $\clD(K|k)$. In particular, the property {\rm\Srj} 
for dominant morphisms of proper valuation-regular-like 
$k$-varieties is \textit{\textbf{birational}}. 
\end{theorem} 
The above Theorem~\ref{thm0} is proved in
section~4, as a consequence of Theorem~\ref{thm3},
and the more general Theorem~\ref{thm1} below, 
which considers the property \Srj\ for morphisms 
of {\it general varieties over number fields.\/} In order 
to announce the latter result, we introduce notation 
and terminology as follows: Let $f:X\to Y$ be a 
morphism of arbitrary varieties over an arbitrary 
base field $k$, and let $X_y$ be the reduced fiber 
of $f$ at $y\in Y$. For $y\in Y$ and $x\in X_y$, we 
denote $L_y\!:=\kappa(y)$, $K_x\!:=\kappa(x)$,
hence $f$ defines canonically a $k$-embedding  
of function fields $K_x\md L_y$. In particular, 
one has the canonical restriction map 
$\clD(K_x|k)\to\clD(L_y|k)$, $v_x\mapsto w_y:
=v_x|_{L_y}$, and to simplify notation, we set 
$\,l_y\!:=L_yw_y$ and $\,k_x\!:=K_xv_x$, hence
$\clO_{v_y}\hra\clO_{v_x}$ gives rise to the  
canonical residue field $k$-embedding $k_x\md l_y$. 
\vskip5pt
%
We say that $w_y\in\clD(L_y|k)$ is \defi{pseudo-split} 
under $f$, if for every $\sigma\in G_{l_y}$ there are 
$x\in X_y$ and $v_x\in\clD(K_x|k)$ satisfying: 
$w_y=v_x|_{L_y}$, $e(v_x|w_y)=1$ if $w_y$ is 
non-trivial, and $k_x\!\otimes_{l_y}\hb3\oli{\,l\,}\!_y$ 
has a factor which is a field stabilized by $\sigma$. 
Further, we say that $y\in Y$ is pseudo-split under $f$
if all $w_y\in\clD(L_y|k)$ are pseudo-split under $f\hb1$, and 
that $f$ is pseudo-split if all $y\in Y$ are pseudo-split.
%
Finally consider the following hypothesis:
\begin{displaymath}
\ha{10}\Bstar \ha{65} 
\left. \begin{array}{c}
$\hb8$\hbox{\it $f:X\to Y$ is a pseudo-split 
morphism of $k$-varieties.\ha{110}}
\end{array} \right.
\end{displaymath}
\begin{theorem} 
\label{thm1}
Let $f:X\to Y$ be a morphism of arbitrary varieties over 
a number field~$k$. Then $f$ satisfies hypothesis \Bstar\ \ 
iff \ $f$ has property {\rm\Srj}. 
\end{theorem} 
We will prove actually a more general result, see 
Theorem~\ref{thm2} in section~3. The main point in 
our approach is to use Ax--Kochen--Ershov 
Principle (AKE) type results (together with 
some general model-theoretical principles about
rational points and ultraproducts of local fields), 
as originating from \cite{Ax, A-K1, A-K2}, 
see e.g.\ \cite{P-R} for details on AKE. Moreover, 
a weak form of AKE in positive characteristic, 
see hypothesis \defi{(qAKE)}$_{\Sgm k}$ after 
Fact~\ref{fkt3} below, implies that~\Bstar\ 
suffices for \Srj\ to hold. To the contrary, 
\cite{Df2} and \cite{LSS} are based on quite 
deep desingularization results, building on previous 
results and ideas, see e.g.\ \cite{Df1, L-S, Sk} 
aimed\ha2---\ha1among other things\ha1---\ha1at 
giving arithmetic geometry proofs of AKE.
\vskip5pt
Here is an enlightening example\ha2---\ha2pointed
out to me by \nmnm{Daniel Loughran}, where the 
above Theorem~\ref{thm1} applies, but the situation
is not covered by the previous methods. 
\begin{example}
Let $k=\lvQ$, $R=k[t]$, 
$X=\Proj R[T_0,T_1,T_2]/(T_0^2+T_1^2-t^2T_2^2)$,
$Y=\Spec R$. One checks directly that the 
canonical projection $f:X\to Y$ has the property 
\Srj, and $f$ is smooth and split above $y\in Y\!$ 
for $y\neq(t)$. But the point $x=(t_0,t_1,t)\in X$ 
above $y=(t)\in Y$ is not v.r.l., hence this situation 
is not covered by previous work. On the other 
hand, $f$ satisfies hypothesis~\Bstar\ha1: Namely, 
all $y\neq(t)$ are split under $f$, thus quasi-split 
under $f$; and for $y=(t)$ one~has 
$X_y\ni x=(t_0,t_1,t)\mapsto (t)=y\in Y\!$, 
$K_x=k=L_y$,~and $\clD(K_x|k)=\{v^0_k\}=\clD(L_y|k)$
consists of the trivial valuation $v^0_k$ of $k$ only. 
Hence $y$ is pseudo-split under $f$ in the 
sense defined above. 
\end{example}
\noindent
{\small{\bf Acknowledgements}. I would 
like to thank several people at the IHP 
special programs in the spring 2018 and 
summer 2019 (especially Zoe Chatzidakis, 
J.-L.\ha3Colliot-Th\'el\`ene, E.\ha3Hrushovski,
J.\ha3Koenigsmann, F.\ha3Loeser, D.\ha3Loughran,
B.\ha3Poonen, A.\ha3Skorobogatov, Arne 
Smeets, S.\ha3Starchenko, Tam\'as Szamuely, 
O.\ha3Witten\-berg) for discussions 
concerning the subject and content of this note.} 
\section{Notations and Basic Facts} 
\subsection{Abstract approximation results 
for points}$\ha0$
\vskip5pt
\noindent
We begin by recalling a few facts, which are/might be
well known to experts. See e.g.\ \cite{B-S}, \cite{Ch},
\cite{F-J}, Ch.7, for details on ultraproducts and 
other model theoretical facts.
\begin{fact}
\label{fkt1}
{\it Let $(k_i\md k)_{i\in I}$ be a family of 
field extensions, $\clP_I$ be a fixed prefilter on 
$I$, and for every ultrafilter $\,\clU$ on $I$ with 
$\,\clP_I\subset\clU$, let $\Ust k:=\prod_{i\in I}k_i/\,\clU$ 
be the corresponding ultraproduct. Then for every
morphism $f:X\to Y$ of $\,k$-varieties, the 
following are equivalent:
\vskip2pt
\itm{30}{
\item[{\rm i)}] There is $I_0\in\clP_I$ such that 
the map $f^{k_i}:X(k_i)\to Y(k_i)$ is surjective 
for all $i\in I_0$.
\vskip2pt
\item[{\rm ii)}] The map $f^{\Ust k}:X(\Ust k)\to Y(\Ust k)$
is surjective for all ultrafilters $\,\clU$.
} 
In particular, if $I$ is infinite, then 
$f^{k_i}:X(k_i)\to Y(k_i)$ is surjective 
for almost all $i\in I\,$ if and only if 
$\,f^{\Ust k}:X(\Ust k)\to Y(\Ust k)$ is surjective 
for all non-principal ultrafilters $\,\clU$ in $I$.
}
\end{fact}
\begin{proof} 
To i) $\Rightarrow$ ii)\ha1: To simplify
notation, we can suppose that $I=I_0$, 
or equivalently, $f^{k_i}:X(k_i)\to Y(k_i)$ 
is surjective for every $i\in I$.
Let $\clU$ be an ultrafilter on $I$ with 
$\clP_I\subset\clU$, and $\Ust y\in Y(\Ust k)$ 
be defined by $\kappa(y)\hra\Ust k$ for 
some $y\in Y$. Let $V\subset Y$ be an 
affine open neighborhood of $y$, say 
$k[V]=k[{\bm u}]=:S$ with ${\bm u}:=(u_1,\dots, u_n)$
a system of generators of the $k$-algebra $S$. 
Then by mere definitions, there is a system 
${\bm u}_\mclU$ of $n$ elements of $\Ust k$ such
that $\Ust y$ is defined by the morphism 
of $k$-algebras
\[
\Ust\psi: S \to S/y \hra \Ust k,\quad
{\bm u}\mapsto{\bm u}_{\mclU}.
\]
Hence, $\clU$-locally, there exist systems ${\bm u}_i$ of 
$n$ elements of $k_i$ and morphisms of $k$-algebras
\[
\psi_i: S \to  S/y \to k_i, \quad {\bm u}\mapsto{\bm u}_i,
\] 
defining $\Ust\psi$, i.e., ${\bm u}_\mclU=({\bm u}_i)_i/\clU$,
and let $y_i\in Y(k_i)$ be the $k_i$-rational point 
defined by $\psi_i$. 
\vskip2pt
Finally, let $(U_\alpha)_{\alpha}$, 
$U_\alpha=\Spec R_\alpha$, be a finite open 
affine covering of $f^{-1}(V)\subset X$. Then
$X(k_i)=\cup_\alpha\, U_\alpha(k_i)$ for all $k_i$,
and $y_i\in\cup_\alpha\, f\big(U_\alpha(k_i)\big)$ 
for every $i\in I$. Since $(U_\alpha)_\alpha$ is 
finite, there exists some $U\!:=U_{\alpha_0}$
such that $\clU$-locally one has: 
$y_i\in f\big(U(k_i)\big)$. Equivalently,
$\clU$-locally, there exists $x_i\in U(k_i)$
such that $f^{k_i}(x_i)=y_i$. Let $R:=k[U]$ be
the $k$-algebra of finite type with $U=\Spec R$. 
Then $f|_U:U \to V$ is defined by a unique 
morphism $f^{\#}_{UV}: S \to R$ of $k$-algebras, 
and there is a unique $k$-morphism
\[
\phi_i: R \to R/x_i\hra k_i
\]
defining $x_i\in U(k_i)$. Further, the fact 
that $f^{k_i}(x_i)=y_i$ is equivalent to 
$f^{\#}_{UV}\circ \psi_i=\phi_i$. Hence if
$\Ust\phi:R\to\Ust k$ is the $k$-morphism
having $\clU$-local representatives $\phi_i:R\to k_i$,
then one has
\[
\Ust\psi\circ f^{\#}_{UV}=\Ust\phi\,.
\] 
Hence if $\Ust x\in X(\Ust k)$ is the 
$\Ust k$-rational point of $X$ defined by 
$\Ust\phi$, then $f^{\Ust k}(\Ust x)=\Ust y$. 
\vskip2pt
To ii) $\Rightarrow$ i)\ha1: By contradiction, 
suppose that for every $J\in\clP_I$ there exists
$j\in J$ such that $f^{k_j}:X(k_j\to Y(k_j)$ 
is not surjective. Then setting $I':=\{i\in I\md 
f^{k_i} \hbox{ is not surjective}\}$, one has:
$\clP'_I:=\{J\cap I'\md J\in\clP_I\}$ is a 
prefilter on $I'\!$, and since $\clP_I\prec\clP'_I$,
every ultrafilter $\clU'$ on $I'$ containing 
$\clP'_I$ is the restriction $\clU'=\clU|_{I'}$
of an ultrafilter $\clU$ on $I$ containing $\clP_I$.
Hence {\it mutatis mutandis,\/} w.l.o.g., we
can suppose that there is an ultrafilter $\clU$ 
continuing $\clP_I$ and a set $J\in\clU$ such that 
$f^{k_i}$ is not surjective for all $i\in J$. Let 
$(V_\beta)_\beta$ be a finite open affine covering 
of $Y$. Then reasoning as above, there exists some
$V:=V_{\beta_0}$ such that $\clU$-locally one 
has: $V(k_i)\not\subset f^{k_i}\big(X(k_i)\big)$. 
Equivalently, $\clU$-locally, there exists $y_i\in V(k_i)$ 
such that $y_i\not\in f^{k_i}\big(X(k_i)\big)$.
That being said, let $\psi_i: S:=k[V] \to k_i$ be the 
morphism of $k$-algebras defining $y_i\in V(k_i)$, 
and $\Ust\psi:S\to\Ust k$ be the $k$-morphism 
defined by $(\psi_i)_i$. Then $\Ust\psi:S\to\Ust k$ 
defines a $\Ust k$-rational point
$\Ust y\in V(\Ust k)\subset Y(\Ust k)$. Hence
by the hypothesis, there is $\Ust x\in X(\Ust k)$
such that $f^{\Ust k}(\Ust x)=\Ust y$. Let
$y\in V$ and $x\in X$ be such that $\Ust y$
and $\Ust x$ are defined by $k$-embeddings
$\kappa(y)\hra\Ust k$, respectively $\kappa(x)\hra\Ust k$.
Then choosing $U\subset X$ affine open with 
$x\in U$ and $f(U)\subset V$, and setting $R:=k[U]$, 
the following hold: 
\vskip2pt
\itm{30}{
\item[a)] $f|_U:U\to V$ is defined by a unique 
morphism of $k$-algebras $f^{\#}_{UV}:S\to R$.
\vskip2pt
\item[b)] $\Ust x$ is defined by a unique morphism 
of $k$-algebras $\Ust\phi:R\to R/x\to\Ust k$.
\vskip2pt
\item[c)] One has that $\Ust\phi=
               \Ust\psi\circ f^{\#}_{UV}$.
}
Therefore, letting $\phi_i:R\to k_i$ be $\,\clU$-local
representatives for $\Ust\phi$, by the general 
nonsense of ultraproducts, $\clU$-locally one has: 
\[
\phi_i=\psi_i\circ f^{\#}_{UV}\,.
\]
Hence if $x_i$ is the $k_i$-rational point of $X$
defined by $\phi_i: R\to k_i$, it follows that
$f^{k_i}(x_i)=y_i$. Therefore, $\clU$-locally, 
one must have that $y_i\in f\big(X(k_i)\big)$, 
contradiction!
\vskip2pt
Finally, for the last assertion of Fact~\ref{fkt1},
we notice: First, the set $\clP_I$ of all the cofinite 
subsets of $I$ is a prefilter on $I$, and $I'\in\clP_I$ 
iff $I\backslash I'$ is finite. Second, an ultrafilter
$\,\clU$ on $I$ is non-principal iff $\,\clP_I\subset\clU$.
Conclude by applying the equivalence
i) $\Leftrightarrow$ ii) to this situation.
\end{proof}
\begin{definition}
\label{dfn3}
A field $k$-embedding $k'\to l'$ is called 
\defi{quasi-elementary}, if there are field 
$k$-embeddings $k'\to l'\to k''\to l''$ with 
$k''\md k'$ and $l''\md l'$ elementary 
$k$-embeddings.
\end{definition}
\begin{fact}
\label{fkt2}
{\it Let $f:X\to Y$ be a morphism of varieties 
over an arbitrary base field $k$, and let $\clC_f$ 
be the class of all the field extensions $k'|k$  
with $f^{k'}:X(k')\to Y(k')$ surjective, One has: 
\vskip2pt
\itm{25}{
\item[{\rm1)}] $\clC_f$ is an elementary class, 
i.e., $\clC_f$ is closed w.r.t.\ ultraproducts and 
sub-ultrapowers. 
\vskip2pt
\item[{\rm2)}] Let $k'\hra l'$ be a
quasi-elementary $k$-field extension. Then
$k'\in\clC_f$ iff $\,l'\in\clC_f$.
}
}
\end{fact}
\begin{proof} Assertion 1) follows from 
Fact~\ref{fkt1} by mere definition. To 2): We begin 
by noticing that $X(\tl k)\subset X(\tl l)$ for all 
$k$-field extensions $\tl k\subset \tl l$. First, consider
the case $l'\in\clC_f$. Then one has $Y(k')\subset Y(l')=
f^{l'}\big(X(l')\big)\subset f^{k''}\big(X(k'')\big)$, 
hence $Y(k')\subset f^{k'}\big(X(k')\big)$, because 
$k'$ is existentially  closed in $k''\!$. Hence 
finally $Y(k')=f^{k'}\big(X(k')\big)$. Second, let 
$k'\in\clC_f$. Embeddings $k'\hra l'\hra k''\hra l''$ 
as in Definition~\ref{dfn2} imply both: $k''\in\clC_f$, 
by assertion~1) above; and $l'$ is existentially 
closed in $l''\!$. Hence reasoning as in the first 
case, one gets $l'\in\clC_f$. 
\end{proof}
\noindent
\subsection{Ultraproducts of localizations of 
arithmetically significant fields\/}$\ha0$
\vskip5pt
\noindent
We introduce notation and recall well known 
facts. We generalize the context in which the 
conclusion of Theorem~\ref{thm1} holds, finally 
allowing to announce Theorem~\ref{thm2} below.
\vskip7pt
For arbitrary fields $k$, consider sets $\Sgm k$ 
of (equivalence classes of) discrete valuations 
$v$ of $k$ such that for all finite non-empty 
subsets $A\subset k$ one has:
\[
\ \ U\!_A\!:=\{v\in\Sgm k\md A\subset\clO_v^\times\}\neq\vid,
\hbox{ hence } \,\clP_{\Sgm k}\!:=\{U\!_A\md A\!\subset\!k^\times
\,\hbox{finite}\} \ \hbox{\it is a prefilter on $\Sgm k$.}
\leqno{\ (\clP)}
\]
\begin{example} 
\label{xmpl1}
$\ha0$ 
Let $X$ be an integral $S$-variety, where $S$ 
is either $\lvZ$ or a field $k_0$, $X_0\subset X$ 
be the set of regular closed points in $X$, and 
$k\!:=\kappa(X)$ be the function~field of $X$. 
One has:
\vskip4pt
\itm{25}{
\item[1)] $\Sgm k$ satisfies $(\clP)$ iff  $X_{\Sgm k}\!:
=\{x_v\in X\md x_v \hbox{ is the center of } 
v\in\Sgm k\}$ is Zariski dense in $X$.
\vskip3pt
\item[2)] If $X_0$ is Zariski dense, there are  
$\Sgm k$ with $X_{\Sgm k}=X_0$, and 
satisfying the following:
\vskip3pt
\itm{20}{
\item[a)] If $X$ is an integral $\lvZ$-variety,
then $kv=\kappa(x_v)$ for all $v\in \Sgm k$.
\vskip3pt
\item[b)] If $X$ is an integral variety over a
field $k_0$, then $kv=\kappa(x_v)$ for all 
$v\in \Sgm k$.
} }
\end{example}
\begin{notations/remarks}
\label{notarem2}
Given $k$ and $\Sigma_k$ as above, let $k_v$ 
be the completion of $k$ at $v\in\Sgm k$, and
$\clU$ \underbar{always} denote ultrafilters on 
$\Sgm k$ with $\clP_{\Sgm k}\subset\clU$. 
Given $\,\clU$, consider the ultraproducts: 
\[
{\textstyle\Ust k:=\prod_{v}\! k_v\ha1/\,\clU,
\ \ \Ust\clO:=\prod_v\!\clO_v\ha1/\,\clU, \ \
\Ust\eum:=\prod_v\!\eum_v\ha1/\,\clU,
\ \ \Uj\kappa:=\prod_v\! k_v v\ha1/\,\clU}.
\]
%
%
\vskip2pt
\noindent
Then $\Ust\clO$ is the valuation ring of $\Ust k$,
say $\Ust\clO=\clO_{\!\Ust v}$ of the valuation
$\Ust v$, with valuation ideal 
$\eum_{\Ust v}=\Ust\eum$, residue field
$\Ust k \Ust v=\Uj\kappa$, and value group
$\Ust v\Ust k=\prod_v\! vk /\,\clU=\lvZ^{\Sgm k}\!/\,\clU$.
\vskip4pt
\itm{25}{
\item[1)] One has the (canonical) diagonal 
field embedding $\Ust\imath:k\hra\Ust k$, 
and $\Ust v$ is trivial on~$k$ (by the fact 
that $\clP_{\Sgm k}\subset\clU\,$).
\vskip3pt
\item[2)] If $\omega_v\subset\clO_v$ is a set 
of representatives of $kv$, then 
$\Ust\omega:=\prod_v\!\omega_v\subset\Ust\clO$
is a system of representatives of $\Ust k\Ust v$
inside $\Ust\clO$. In particular, if $\omega_v$
are multiplicative, so is $\Ust\omega$.
\vskip3pt
\item[3)] The value group $\Ust v\Ust k$ is a 
$\lvZ$-group. Further, if $\pi_v\in k_v$ is a 
uniformizing parameter for $v\in\Sgm k$, then 
$\Uj\pi=(\pi_v)_v/\,\clU$ is an element 
of minimal value in $\Ust v\Ust k$. 
\vskip3pt
\item[4)] The field $\Ust k$ is Henselian with
respect to $\Ust v$, and one has:
\vskip3pt
\itm{20}{
\item[a)] Suppose that $\chr(k)=0$. Then 
$\Ust v$ is trivial on $\lvQ\subset\Uj\kappa$,
and if $\clT\subset\Ust\clO$ is any lifting of a 
transcendence basis of $\Uj\kappa\md\lvQ\,$,
by Hensel Lemma one has: The relative 
algebraic closure $\Uj\kappa'\subset\Ust\clO$ 
of $\lvQ(\clT)$ in $\Ust k$ is a field of 
representatives for $\Uj\kappa$.
\vskip4pt
\item[b)] The fields of representatives
 $\Uj\kappa'\!\subset\Ust\clO$ for $\Uj\kappa$ 
are relatively algebraically closed~in~$\Ust k$. 
\vskip3pt
\item[c)] For $\Uj\kappa'\subset\Us k$ as above, let 
$\Us k:=\Uj\kappa'(\Uj\pi)^h$ be the Henselization 
of $\Uj\kappa'(\Uj\pi)$ w.r.t.\ the $\Uj\pi$-adic 
valuation, and set $\Us v:=\Ust v|_{\Us k}$. Then 
one has $k$-embeddings: 
\[
\ha{30}\Us k:=\Uj\kappa'(\Uj\pi)^h \hra\Ust k \ \
\hbox{\it as valued fields, and } \ 
\kappa(\Us v)\,=\Uj\kappa=\,\kappa(\Ust v),
\]
and further, $\Us v\Us k=\lvZ\hra 
\lvZ^{\Sgm k}/\,\clU=\Ust v \Ust k$ are 
$\lvZ$-groups having $\Uj\pi$ as the element 
of minimal positive value. Hence the 
Ax--Kochen--Ershov Principle (AKE) 
implies:
} 
}
\end{notations/remarks}
\begin{fact} 
\label{fkt3}
{\it If $\,\chr(k)=0$, then $\Us k\hra \Ust k$ is an 
elementary $k$-embedding of (valued) fields.\/}
\end{fact} 
\vskip2pt
\noindent
Unfortunately, if $\chr(k)=p>0$, it is not known 
whether the conclusion of Fact~\ref{fkt3}~holds. 
Therefore, for $\clU$ on $\Sgm k$ as in 
Remarks/Notations \ref{notarem2} above, and 
the corresponding $k$-embeddings 
$\Us k=\kappa_\UU(\Uj\pi)^h\hra\Ust k$, 
consider the following hypothesis\ha2---\ha1which 
is weaker than the Ax--Kochen--Ershov Principle
(AKE), holding if $\chr(k)$ is zero:
\vskip8pt
\noindent
\centerline{\defi{(qAKE)$_{\Sgm k}\ha{20}$} 
\ {\it $\Uj k\to \Ust k$ are quasi-elementary 
$k$-embeddings for all $\,\clU$.\ha{80}}
} 
\vskip8pt
\noindent
In the above notation, one has, see e.g.\ 
\cite{Ch}, and~\cite{F-J},~Ch.\ha211\ha2:
%
%
\begin{fact}[{\bf Residue fields}]
\label{fkt4}
{\it Let $k$ be as in {\rm Example~\ref{xmpl1},~2)}. 
\vskip2pt
{\rm1)} \ In case~{\rm a)}, $\Uj\kappa$ 
is a perfect $\aleph_1$-saturated PAC quasi-finite 
field.\footnote{\ha2\rm That is, the absolute 
Galois group $G_{\Uj\kappa}$ of $\Uj\kappa$ 
is $\,G_{\Uj\kappa}\cong\widehat\lvZ\cong G_{\lvF}$, 
where $\lvF$ is any finite field.}
\vskip2pt
{\rm2)} \ In case~{\rm b)}, let $k_0$ be 
(perfect) PAC. Then $\Uj\kappa$ is (perfect) PAC 
and $\aleph_1$-saturated.
}
\end{fact}

\noindent
Note that if $\chr(k)=p>0$, then in case~1) above 
one has: If $\lvF_v:=\oli\lvF_p\cap k_v$, then 
$k_v=\lvF_v\lps{\pi_v}$ for any $\pi_v\in k$ 
with $v(\pi_v)=1$. Hence $\Uj\lvF:=
\prod_v\lvF_v/\,\clU\subset\Ust\clO$ is a perfect
field and a system of representatives for $\Uj\kappa$. 
Further, if $t\in k$ is part of a separable transcendence 
basis for $k\md \lvF_p$, then $U_t:=\{v\md v(t)=0\}$
lies in $\clU$, and if $a_v=tv\in k_vv=\lvF_v$ is the 
residue of $t$ at $v$, then $\pi_v:=t-a_v\in k_v$ 
satisfies $v(\pi_v)=1$, and one has a canonical embedding:
\[
k\hra \lvF_v\lps{\pi_v}=k_v.
\] 
Thus setting $a_\UU:=(a_v)_v/\,\clU\in\kappa_\UU$, 
one has $\pi_\UU=t- a_\UU\in k\,\lvF_\UU$. But 
despite of these special/particular facts, it is 
unknown whether the conclusion of Fact~\ref{fkt3}
holds in this case.  
\subsection{Generalized pseudo-split extensions}$\ha0$
\vskip5pt
\noindent
We begin by discussing the case of fields $k$ 
as in Example~\ref{xmpl1},~2),~a). Precisely, 
$k=\kappa(X)$ is the function field of an integral 
$\lvZ$-variety $X$, and further: $X_0\subset X$ 
is the set of closed regular points (which is 
Zariski open in the set of all closed points), and
$X_{\Sgm k}\subset X_0$ is Zariski dense.
\vskip2pt
We say that $\sigma\in G_k$ and the co-procyclic 
extension $\oli k^{\ha1\sigma}|\ha1 k$ of $k$ 
are \defi{$\Sgm k$-definable}, if for all 
finite Galois extensions $l\md k$, and all
$U_A\in\clP_{\Sgm k}$, one has:
\[
U_{A,\,l|k}(\sigma):=\{v\in U\hb3_A\md 
\hbox{$v$ unramified in $l\md k$ and
$\Frob(v)=\sigma|_l$}\}\neq\vid.
\]
Notice that in the case $X_{\Sgm k}\subset X_0$
and $kv=\kappa(x_v)$ for all $v\in\Sgm k$ one has:
\vskip2pt
\itm{15}{
\item[-] If $X_{\Sgm k}$ has Dirichlet density 
$\delta(X_{\Sgm k})=1$, e.g.\ if $X_{\Sgm k}
\subset X_0$ is Zariski open, it follows 
by the Chebotarev Density Theorem, see 
e.g.~\nmnm{Serre}~\cite{Se1}, that all 
$\sigma\in G_k$ are $\Sgm k$-definable.
\vskip2pt
\item[-] If $\Sgm k$ is \defi{Frobenian} in the
sense of \nmnm{Serre}~\cite{Se2}, 3.3, say
defined by a finite Galois extension $l\md k$
and a set of conjugacy classes $\Phi\subset 
{\rm Gal}(l\md k)$, then $\sigma\in G_k$ is 
$\Sgm k$-definable iff $\sigma|_l\in\Phi$.
} 
\begin{fact}
\label{fkt4-5}
{\it In the above notation, $\sigma\in G_k$ is 
$\Sgm k$-definable iff $\,\,\oli k ^\sigma\hb3=
\Ust k\cap\oli k\,$ for some $\,\clU$.\/}
\end{fact}
\begin{proof}
For the direct implication,  notice that 
$\clP_{\Sgm k}(\sigma):=\{U_{A,\,l|k}\}_{A,\,l|k}$ 
is a prefilter on $\Sgm k$ such that any ultrafilter 
$\clU$ containing $\clP_{\Sgm k}(\sigma)$ 
contains $\clP_{\Sgm k}$. Let $l\md k$ be a 
finite Galois extension. Then for 
$v\in U\hb3_{A,\,l|k}(\sigma)\in\clU$, setting 
$l_v:=lk_v$ one has: $l_v|k_v$ is unramified 
and $l^\sigma=l\cap k_v$. Hence
$l^\sigma=l\cap\Ust k$, and finally 
$\oli k^\sigma=\oli k\cap\Ust k$. 
\vskip2pt
Conversely, let $\clU$ be such that 
$\,\,\oli k^\sigma\hb3=\Ust k\cap\oli k$.
To show that $\sigma$ is $\Sgm k$-definable,
we have to show that all the sets 
$U\hb3_{A,\,l|k}(\sigma)$ are non-empty. First, 
since $\,\,\oli k^\sigma\hb3=\Ust k\cap\oli k$, it
follows that for every finite Galois extension
$l\md k$, one has $l^\sigma=\Ust k\cap l$. 
Hence for every $l\md k$ there exists a set 
$V_l\in\clU$ such that 
for all $v\in V_l$ one has $l^\sigma=k_v\cap l$.
Further, let $U\hb3_A\subset\Sgm k$ be given. 
Since $\clP_{\Sgm k}\subset\clU$, hence 
$U\hb3_A\in\clU$, w.l.o.g., we can suppose 
that $V_l\subset U\hb3_A$. Second,
let $B\subset k^\times$ be a finite set such 
that all discrete valuations $w$ of $k$ with
$w(B)=0$ are unramified in $l\md k$. (Note
that such sets $B$ exist: If $X_l\to X$ is
the normalization of $X$ in the finite Galois
extension $l\md k$, then there exists an affine
open dense subset $X'\subset U$ such that
$X_l$ is \'etale above $X'\!$. Hence if 
$w$ has its center in $X'\!$, then $w$ is 
unramified in $l\md k$, etc.)
Then setting $A_l:=A\cup B$, one has: 
$V_l\cap U\hb3_{A_l}\in\clU$, and all
$v\in V_l\cap U\hb3_{A_l}$ are unramified
in $l\md k$. Hence $U\hb3_{A_l,\,l|k}\neq\vid$,
thus $U\hb3_{A,k|l}\supset U\hb3_{A_l,\,l|k}$
is non-empty as well, thus concluding that 
$\sigma$ is $\Sgm k$-definable.
\end{proof}
\begin{definition}
For $k$ as above, let $M\md N\md k$ be field
extensions, and $N'\md N$ be algebraic. 
\itm{25}{
\item[1)] $N'\md N$ is called co-procyclic 
$\Sigma_k$-definable, if $N'=\oli N^{\,\sigma_N}$ 
for some $\sigma_N\in G_N:=\Aut{_N}{\oli N}$ 
which is itself $\Sigma_k$-definable, i.e.,  the 
restriction $\sigma:=\sigma_N|_{\oli k}\in G_k$ is 
$\Sgm k$-definable. 
\vskip2pt
\item[2)] $M\md N$ is called $N'$-split, 
or split above $N'\!$, if the $N'$-algebra 
$M\otimes_NN'$ has a factor $M'$ which is a field,
hence $M'\md N'$ is a field extension.
}
\end{definition}
%
%
\begin{proposition}
\label{prp1}
Let $M\md N$ be function fields over $k$ as 
in {\rm Example \ref{xmpl1},~2),~a)}. One has: 
\vskip2pt
\itm{25}{
\item[{\rm 1)}] An algebraic extension $N'|N$ 
is co-procyclic $\Sgm k$-definable if and
only if there is $\,\clU$ and a $k$-embedding 
$N\hra\Uj\kappa$ such that $N'=\oli N\cap\Uj\kappa$.
\vskip2pt
\item[{\rm 2)}] Let $N'=\oli N\cap\Uj\kappa$ 
as above be given. Then $M\md N$ is split above 
$N'$ iff $\,M\md N$ is separably generated and 
$N\hra\Uj\kappa$ prolongs to a field embedding 
$M\hra\Uj\kappa$.
}
\end{proposition}
\begin{proof} To 1): To the direct implication: 
Since $\kappa_\UU$ is a 
perfect pseudo-finite field, $k\hra N\hra \Uj\kappa$
gives rise to embedding of perfect fields 
$k'=\oli k\cap\Uj\kappa\hra N'=\oli N\cap\Uj\kappa
\hra\Uj\kappa$ and to surjective projections 
$\widehat\lvZ\cong G_{\kappa_\UU}\srjr G_{N'}
\srjr G_{k'}$. Hence $N'|\ha1N$ is by mere 
definitions co-procyclic and $\Sgm k$-definable. 
For the converse implication, let $N'|\ha1N$ be 
co-procyclic and $\Sgm k$-definable. Then 
$k':=\oli k\cap N'$ is obviously co-procyclic 
and $\Sgm k$-definable. Hence, there is some 
$\clU$ such that $k'=\oli k\cap\kappa_\UU$, 
and obviously, $N'|k'$ is a regular field 
extension. We claim that there is a $k$-embedding 
$N\hra\kappa_\UU$ 
such that $N'=\oli N\cap\kappa_\UU$, hence 
$k'\subset N'$. First, $N'_0:=Nk'\subset N'$ 
is a regular function field over $k'\!$, and 
setting $\tilde N_0=N'_0$, 
there is an increasing sequence of cyclic field 
subextensions $(\tilde N_i|N'_i)_{i\in\lvN}$ 
of $\oli N|N'$ such that $N'=\cup_{i\in\lvN}N'_i$, 
$\oli N=\cup_{i\in\lvN}\tilde N_i$, and 
$\tilde N_i|N'_i$ is the maximal subextension
of $\oli N|N'$ of degree $\leqslant i$. By algebra 
general non-sense, the sequence $(\tilde N_i|N'_i)_i$
and the conditions it satisfies are expressible
by a type $p(\bm t)$ over $k'$, where $\bm t$
is a transcendence basis of $N_0|k'$; and since
$\kappa_\UU$ is a perfect PAC quasi-finite 
field, the type $p(\bm t)$ is finitely satisfiable.
Thus $\kappa_\UU$ being $\aleph_1$-saturated, 
the type $p(\bm t)$ is satisfiable in $\kappa_\UU$,
thus $N=N_0$ has a $k'$-embedding
$N\hra\kappa_\UU$ such that $N'=\oli N\cap\kappa_\UU$.
\vskip2pt
To 2): For the direct implication, let $M'$ 
be a factor of $M\otimes_N\hb1N'$ such that 
$M'|N'$ is a regular field extension. Since 
$N'|\ha1N$ contains the prefect closure of $N$,
it follows that $M\md N$ must be separably 
generated (because otherwise all the factors 
of $M\otimes_N N'$ have non-trivial nilpotent 
elements). Hence $M=N(Z_N)$ for integral
$N$-variety $Z_N$ such that $Z_N\times_N N'$
has a geometrically integral irreducible
component $Z_{N'}$ of multiplicity one
with $M'=N'(Z_{N'})$. Since $\kappa_\UU$ is a 
$\aleph_1$-saturated PAC (quasi-finite) field,
$Z_{N'}(\kappa_\UU)$ contains ``generic points'' 
of $X_{N'}$, that is, $M'$ is $N'$-embeddable 
into $\kappa_\UU$. 
\vskip2pt
For the reverse implication, since $M\md N$ 
is separably generated, it follows that
$M\otimes_N N'$ is a product of fields. 
Let $M\hra\kappa_\UU$ be a prolongation 
of $N\hra\kappa_\UU$. Then 
\[
N':=\oli N\cap\kappa_\UU\hra
\oli M\cap\kappa_\UU=:M'\hra\kappa_\UU
\]
are co-procyclic extensions, and $M\otimes_NN'$ 
has a factor $M_{N'}$ which is $N'$-embeddable in 
$M'$. Since $N'$ is perfect, $N'=\oli N\cap M'\hra M'$
is regular, hence $M_{N'}|N'$ is regular.  
\end{proof}
\noindent
The Proposition~\ref{prp1} above hints at the 
following generalization of the pseudo-splitness:
\begin{definition}
\label{dfn2}
In Notations/Remarks~\ref{notarem2}, let 
$M\md N$ be a $k$-field extension, and 
$\jmath:N\to\Uj\kappa$ be a field $k$-embedding.
\vskip2pt
\itm{25}{
\item[1)] An algebraic $k$-field extension 
$N'\md N$ is called \defi{$\jmath$-definable}, 
if $N'$ is isomorphic to $\oli N\cap\Uj\kappa$ 
as $N$-field extensions. To simplify notation, 
we write $N'=\oli N\cap\Uj\kappa$. 
\vskip2pt
\item[2)] For $\jmath:N\hra\Uj\kappa$ defining
$N'\md N$ as above, $M|N$ is called \defi{\jspl}, 
if $M\md N$ is separably generated, and $\jmath$ 
prolongs to $M\hra\Uj\kappa$. 
}
\end{definition}
%
%
\begin{proposition}
\label{prp2}
Let $M|N$ be an extension of $k$-function 
fields over $k$. Let $\jmath:N\hra\Uj\kappa$ be a 
$k$-embedding, and $N'=\oli N\cap\Uj\kappa$ 
be a $\jmath$-definable extension of $N$. One has:
\vskip2pt
\itm{25}{
\item[{\rm 1)}] Let $M=N(Z_N)$ with $Z_N$ 
an integral $N$-variety. Then $M\md N$ is \jspl\ 
iff $\,Z_N\!\times_{\!N}\! N'$ is geometrically 
reduced and $Z_N(\Uj\kappa)$ is Zariski dense.
\vskip2pt
\item[{\rm 2)}] In particular, for a 
$N'=\oli N\cap\Uj\kappa$ as above, 
the following hold:
\vskip2pt
\itm{20}{
\item[{\rm a)}] If $\Uj\kappa$ is PAC, 
$M\md N$ is $N'$-split iff $\,M\!\otimes_N\! N'$ 
factor $M'$ with $M'|\ha1N'$ regular. 
\vskip2pt
\item[{\rm b)}] If $\chr(k)=0$, $M\md N$ is 
$N'$-split iff $\,\jmath\!:\!N\!\hra\!\Uj\kappa$ 
has a prolongation $M\hra\Uj\kappa$. 
} 
}
\end{proposition}
\begin{proof} 
To 1): The implication $\Rightarrow$ is 
simply a reformulation in terms of algebraic 
geometry of the fact that $M|N$ is $N'$-split. 
For the converse implication, one has: First, 
$Z_{N'}\!:=Z_N\times_N N'$ being reduced, 
its ring of rational functions is the product of 
the function fields $M'_\alpha:=N'(Z'_\alpha)$
of the irreducible components $Z'_\alpha$
of $Z_{N'}$. Second, since $Z_N(\Uj\kappa)$ 
is Zariski dense, $Z'_\alpha(\Uj\kappa)$ is
Zariski dense for some $\alpha$. And since
$\Uj\kappa$ is $\aleph_1$-saturated, by general 
model theoretical non-sense, $Z'_\alpha(\Uj\kappa)$ 
contains ``generic points'' of the $N'$-variety
$Z'_\alpha$. Finally, each such point defines an 
$N'$-embedding $M'_\alpha=N'(Z_\alpha)\hra\Uj\kappa$, 
which prolongs $\jmath:N\hra \Uj\kappa$.
\vskip2pt
To~2): First, the implication $\Rightarrow$ is the
same as in assertion~1. The converse implication
in case b) is clear, and in case a) it follows from 
assertion~1): Since $\Uj\kappa$ is a PAC field, and 
$Z_{N'}$ is a geometrically integral $N'$-variety, 
it follows that $Z_{N'}(\Uj\kappa)$ is Zariski 
dense, etc.
\end{proof}
\begin{corollary}[\bf Example~\ref{xmpl1} revisited]
\label{xmpl2}
Let $k$ and $\Sgm k$ be as in {\rm Example
\ref{xmpl1},~2).} Let $N\md k$ be a function 
field over $k$, and $N'\md N$ be $\jmath$-definable. 
A $k$-extension of function fields $M\md N$ 
is $N'$-split iff $M\otimes_NN'$ 
has a factor $M'$ such that  $M'\md N'$ is a 
regular field extension.
\end{corollary}
\section{Proof of (Generalizations of) 
      Theorem~\ref{thm1}}
\subsection{Setup for a generalization of 
Theorem~\ref{thm1}\/}$\ha0$
\vskip5pt
The generalization of Theorem~\ref{thm1} we 
aim at is based on generalizing hypothesis~\Bstar, 
i.e., the notion of {\it pseudo-split morphism\/} 
$f:X\to Y\!$, as already hinted at in Definition~\ref{dfn2}.
In order to do so, we begin by recalling the 
following obvious facts concerning splitness.
\vskip2pt
First, let $M\md N$ be a field extension, and 
$N'\md N$ be an algebraic extension with  $N'$ 
perfect. Then the following are equivalent:
\vskip2pt
\itm{30}{
\item[i)] $M\md N$ is $N'$-split (and if so, 
$M\!\otimes_N\!N'$ has regular field extension
$M'|N'$ as a factor).
\vskip2pt
\item[ii)] $M\md N$ is separably generated, and 
$\oli N\cap M$ is embeddable in $N'$.
} 
\vskip2pt
Second, let $M\md N$ be $N'$-split, $L\md M$ 
be $M'$-split. The following {\it transitivity of 
splitness\/} holds:
\vskip2pt
\itm{30}{
\item[a)] If $N'=M'\cap\oli N$, then $L\md N$ is 
split above $N'\!$.
\vskip2pt
\item[b)] If $\tilde M\md N\hra M\md N$ and 
$\tilde N'\md N\hra N'\md N$ are subextensions, 
then $\tilde M\md N$ is $\tilde N'$-split.
}
Next recall that for morphisms $f:X\to Y$
of $k$-varieties, the reduced fiber $X_y\subset X$
at $y\in Y$, and $x\in X_y$, 
we set $L_y:=\kappa(y)\hra\kappa(x)=:K_x\,$. 
In particular, one has the canonical restriction
map $\clD(K_x|k)\to\clD(L_y|k)$, and for 
$v_x\in\clD(K_x|k)$ and $w_y:=v_x|_{L_y}$, 
one has the canonical $k$-embedding of residue 
function fields $l_y:=L_yw_y\hra K_xv_x=:k_x$.
\begin{definition}
\label{dfn4}
Let $k$, $\Sgm k$, and $\,\clU$ be as in 
Remarks/Notations~\ref{notarem2}. Recalling
Definition~\ref{dfn2}, for morphisms of 
$k$-varieties $f:X\to Y\!$, define:
\vskip2pt
\itm{25}{
\item[1)] $w_y\in\clD(L_y|k)$ is 
$\Sgm k$-pseudo-split under $f\hb1$, 
if for all $\clU$ and all $k$-embeddings 
$\jmath:l_y\hra\Uj\kappa$ there is 
$v_x\in\clD_{w_y}$ such that $k_x\md l_y$ 
is \jspl, and $e(v_x|w_y)=1$ if $w_y$ is non-trivial. 
\vskip2pt
\item[2)] We say that $f$ is \defi{$\Sgm k$-pseudo-split} 
if all $w_y\in\clD(L_y|k)$, $y\in Y\!$, are 
$\Sgm k$-pseudo-split under~$f\hb1$.
}
\end{definition} 
\noindent
Finally, the generalization of hypothesis~\Bstar\ 
we were hinting at is the following hypothesis:
\begin{displaymath}
\ \hbox{\Bstar$^\nix_{\Sgm k}$} \ha{55} 
\left. \begin{array}{c}
$\hb8$\hbox{\it $f:X\to Y$ is a $\Sgm k$-pseudo-split 
morphism of $k$-varieties.\ha{90}}
\end{array} \right.
\end{displaymath}
\noindent
Correspondingly, the natural generalization of
\Srj\ from Introduction is the property:
\begin{displaymath}
\ \hbox{\Srj$_{\Sgm k}$} \ha{20}
\left. \begin{array}{c}
\hbox{\it $f^\kv:X(\kv)\to Y(\kv)$ is surjective for 
all $v\in U\hb3_A$ for some $A\subset k^\times\!$.\ha{100}}
\end{array} \right.
\end{displaymath}
\noindent
Notice that for number fields $k$ and 
$\Sgm k=\lvP(k)$ one has: The hypotheses 
\Bstar\ and \Bstar$_{\Sgm k}$ are equivalent, 
and so are properties \Srj\ and \Srj$_{\Sgm k}$.
Further, \defi{(qAKE)$_{\Sgm k}$} holds (by 
the usual Ax--Kochen--Ershov Principle).
Hence Theorem~\ref{thm1} follows from
the more general:
\begin{theorem}
\label{thm2}
In Notations$\,/$Remarks~\ref{notarem2}, 
let $k$ endowed with $\Sgm k$ satisfy
{\rm\defi{(qAKE)$_{\Sgm k}$}}. Then for 
a morphism $f:X\to Y$ of $k$-varieties the 
following hold:
\vskip2pt
\itm{25}{
\item[{\rm 1)}] If $f$ satisfies hypothesis 
\Bstar$_{\Sgm k}$, then $\,f$ has property 
{\rm \Srj$_{\Sgm k}$}.
\vskip2pt
\item[{\rm 2)}] Let $\chr(k)=0$. Then $\,f\,$ 
satisfies hypothesis \Bstar$_{\Sgm k}\,$ iff 
$\,f\,$ has property {\rm \Srj$_{\Sgm k}$}.
} 
\end{theorem}
The proof of Theorem~\ref{thm2} is reduced
to proving the Key Lemma~\ref{thekeylmm}
below as follows: First, by Fact~\ref{fkt1}, 
the property \Srj$_{\Sgm k}$ is equivalent to 
$f^{\Ust k}:X(\Ust k)\to Y(\Ust k)$ being surjective 
for all~$\,\clU$. Second, by Fact~\ref{fkt2} 
combined with Fact~\ref{fkt3}, and the 
hypothesis~\defi{(qAKE)$_{\Sgm k}$}, the 
surjectivity of $f^{\Ust k}$ is equivalent to the 
surjectivity of $f^{\Us k}:X(\Us k)\to Y(\Us k)$. 
Hence the property \Srj$_{\Sgm k}$ is equivalent 
to the following condition in terms of ultrafilters: 
\vskip5pt
\noindent \
\Srj$_\UU\ha{20}$ {\it $f^{\Us k}:X(\Us k)\to Y(\Us k)$ 
is surjective for every $\,\clU$.\/}
\vskip5pt
\noindent
This reduces the proof of Theorem~\ref{thm2}
to proving the following:
\begin{keylemma}
\label{thekeylmm}
Let $k$, $\Sgm k$ be as in
Notations$\,/$Remarks~\ref{notarem2}, and 
hypothesis {\rm\defi{(qAKE)$_{\Sgm k}$}} 
be satisfied. Then for a morphism $f:X\to Y$ 
of $k$-varieties the following hold:
\vskip2pt
\itm{25}{
\item[{\rm 1)}]  If $f$ satisfies hypothesis 
\Bstar$_{\Sgm k}$, then $\,f$ has the 
property {\rm \Srj$_\clU$}.
\vskip2pt
\item[{\rm 2)}] Let $\chr(k)=0$. Then $f$ satisfies 
hypothesis \Bstar$_{\Sgm k}$ iff $\,f\,$ has the property 
{\rm \Srj$_\clU$}.
}
\end{keylemma}
\vskip5pt
\noindent
\subsection{Proof of the Key Lemma~\ref{thekeylmm}}$\ha0$
\vskip5pt
\noindent
We begin by recalling basic facts from
valuation theory, which are well known 
to experts. 
\begin{fact}
\label{fkt6}
{\it Let $\Omega,w$ be a Henselian field with
$\chr(\Omega w)=0$. Then every subfield
$l\subset\Omega$ with $w|_l$ trivial is contained
in a field of representatives $\kappa'\subset\Omega$
for $\Omega w$.\/}
\end{fact}
\begin{proof} 
This is a well known consequence of the Hensel 
Lemma.
\end{proof}
\noindent
We next recall basic facts about valuations 
without (transcendence) defect, see 
\cite{BOU},~Ch.\ha2VI, and \cite{Ku},
for some/more details on (special cases of)
this. Let $\Omega,w$ be a 
valued field with $w|_{\kappa_0}$ trivial
on the prime field $\kappa_0$ of $\Omega$. 
One says that $w$ has no \defi{(transcendence) 
defect} if there exists a transcendence basis of 
$\Omega\md\kappa_0$ of the form $\clT_w\cup\clT$ 
satisfying the following: First, $w\clT_w$
is a {\it basis\/} of the $\lvQ$-vector space 
$w\Omega\otimes\lvQ$, and second, $\clT$ 
consists of $w$-units such that its image 
$\clT\!w$ in the residue field $\Omega w$ 
is a {\it transcendence basis\/} of 
$\Omega w\md\kappa_0$. In particular, 
if $\kappa'_\clT\subset\Omega$ is the relative 
algebraic closure of $\kappa_0(\clT)$ in 
$\Omega$, then $\kappa'_\clT$ is a maximal 
subfield of $\Omega$ such that $w$ is trivial
on $\kappa'_\clT$, and further, $\Omega w$ is 
algebraic over $\kappa'_\clT w$. Moreover, 
if $w$ is Henselian, then Hensel Lemma
implies that $\Omega w$ is purely inseparable 
over $\kappa'_\clT w$. 
\vskip2pt
One of the main properties of valuations $w$ 
without defect is that for any subfield 
$N\subset\Omega$, the restriction of $w$ to 
$N$ is a valuation without defect as well,
see~\cite{Ku}. In particular, if $l\subset\Omega$
is any subfield such that $w|_l$ is trivial, and
$N\md l$ is a function field, then $w|_N$ is a
prime divisor of the function field $N|l$ if
and only if $w|_N$ is a discrete valuation. 
\vskip2pt
Hence for the field $\Us k=\Uj\kappa'(\Uj\pi)^h$ 
endowed $\Us v$ from Notations/Remarks
\ref{notarem2},~4),~c),  one has: 
\begin{fact}
\label{fkt7}
{\it Let $l\subset\Us k$ be a subfield with $\Us v$
trivial on $l$. Let $N\md l$ be a function field 
and $N\hra\Us k$ be an $l$-embedding. Then 
$w:=\Us v|_N$ is either trivial, or a prime divisor 
of $N\md l$}
\end{fact}
\begin{proof} This is an immediate 
consequence of the discussion above. 
\end{proof}
\begin{fact}
\label{fkt8}
{\it Let $N^h$ be the Henselization of a function
field $N|l$ w.r.t.\ a prime divisor $w$. Let
$\kappa'\subset\Omega$ be a field of 
representatives for $\Nw$, and $\pi\in N$ 
have $w(\pi)=1$. Then
$N^h=\kappa'(\pi)^h\!$.\/}
\end{fact}
\begin{proof} The Henselian subfield 
$\tilde N:=\kappa'(\pi)^h$ of $N^h$ satisfies 
$\tilde Nw=N^hw$ and $w\tilde N=wN$. 
Since $w$ has no defect, the fundamental 
equality holds. Hence $[N^h\!:\!\tilde N]=
e(N^h|\tilde N)f(N^h|\tilde N)=1$, thus finally 
implying $N^h=\tilde N=\kappa'(\pi)^h\!$.
\end{proof}
\vskip0pt
\noindent
\subsubsection{Proof of assertion {\rm 1)} 
of the Key Lemma \ref{thekeylmm}}$\ha0$
\vskip10pt
Let $\Us y\in Y(\Us k)$ be defined by 
a point $y\in Y$ and a $k$-embedding 
$\Us\jmath:L_y\hra\Us k$. By Fact~\ref{fkt7} 
above, $w:=v_y:=\Us v|_{L_y}\in\clD(L_y|k)$ 
is either trivial or a prime divisor of $L_y|k$, 
and let $\jmath:l_y\hra\Uj\kappa$ be the 
corresponding $k$-embedding of the residue 
fields. Since $f$ is $\Sgm k$-pseudo-split,
there is $x\in X_y$ and $v:=v_x\in\clD(K_x|k)$ 
on $K_x=\kappa(x)$ such that $w=v|_{L_y}$,
the residue field embedding $k_x\md l_y$ is 
\jspl, and $e(v|w)=1$ if $w$ is 
non-trivial. Hence by definitions, $k_x\md l_y$ is 
separably generated, and $\jmath:l_y\hra\Uj\kappa$
has a prolongation $\imath:k_x\hra\Uj\kappa$. Let 
$\clT_0$ be a separable transcendence basis of $k_x$ 
over $l_y$, and $\clT\subset K_x$ be a preimage of
$\clT_0$ under the canonical residue field
projection $\clO_v\to K_xv_x$. One has the
following:
\vskip2pt
\itm{15}{
\item[-] Setting $N:=L_y$ and $M:=K_x$, one has 
$\Nw=l_y$, $k_x=\Mv$, and further: $\clT_0$ is a 
separable transcendence basis of $\Mv$ over 
$\Nw$, and $\clT\subset M$ is a preimage of 
$\clT_0$ under $\clO_v\to\Mv$.
\vskip2pt
\item[-] Set $N_{\clT}:=N(\clT)\subset M$. Since $w=v|_N$, 
it follows by mere definition that $w_\clT:=v|_{N_\clT}$ 
is the Gauss valuation of $N_\clT$ defined by $w$ 
and $\clT\!$.
\vskip2pt
\item[-] Setting $\kappa_N:=\jmath(\Nw)\hra
\imath(\Mv)=:\kappa_M$, it follows that 
$\imath(\clT_0)$ is a separable transcendence 
basis of $\kappa_M$ over $\kappa_N$. 
\vskip2pt
\item[-] Setting $\Us N:=\Us\jmath(N)\subset\Us k$,
let $\Us\clT\subset\Us k$ be a preimage 
of $\imath(\clT_0)$ under $\Us\clO\to\Uj\kappa$, 
and set $N_{\Us\clT}:=\Us N(\Us\clT)$. Then 
the restriction $w_{\Us\clT}$ of $\Us v$ to 
$N_{\Us\clT}$ is the Gauss valuation of $\Us N$ 
defined by $\Us w=\Us v|_{\Us N}$ and 
$\Us\clT$. Hence one has 
a $k$-isomorphism of valued fields
\[ 
\jmath^\nix_{\Us\clT}:N_\clT\to N_{\Us\clT}\subset\Us k.
\]
\item[-] Let $\Us N^h\subset N_{\Us\clT}^h
\subset\Us k$ be the Henselizations of 
$\Us N\subset N_{\Us\clT}$ in $\Us k$. Then since 
$\kappa_M$ is finite separable over the residue field 
$N_{\Us\clT}w_{\Us\clT}=\kappa_N\big(\imath(\clT_0)\big)$, 
one has: There exists a unique algebraic unramified 
subextension $\Us M^0\md N_{\Us\clT}^h$ of 
$\Us k\md N_{\Us\clT}^h$ with residue field 
$\Us M^0\Us v=\kappa_M$.   
}
\vskip2pt
Finally, one has the following case-by-case discussion:
\vskip5pt
\underbar{Case 1}. $v$ is trivial. Then $w$ is 
trivial, hence $N=\Nw\hra\Mv=M$, and 
$\tilde y\in Y(\Us k)$ is defined by the $k$-embedding
$\Us\jmath:\kappa(y)=N\to\Us N\subset\Us k$.  
In particular, in the above notation, the valuations
$w_\clT$ and $w_{\Us\clT}$ are trivial, thus
$N=N^h\hra N_{\Us\clT}^h=N_{\Us\clT}$, and 
$\Us M^0\md N_{\Us\clT}$ is a finite separable 
extension of $N_{\Us\clT}$ such that the residue
map $\Us\clO\to\Uj\kappa$ defines an
isomorphism $\Us M^0\to\kappa_M$. Hence if 
$\imath_0:\kappa_M\to\Us M^0$ is the inverse 
of the isomorphism $\Us M^0\to\kappa_M$, 
one has:
\[
\Us\imath:M\hor{\imath}\kappa_M\hor{\imath_0} 
\Us M^0\subset\Us k
\]
is an isomorphism prolonging $\Us\jmath:N\to\Us k$,
thus defining $\tilde x\in X(\Us k)$ such that 
$f^{\Us k}(\tilde x)=\tilde y$.
\vskip5pt
\underbar{Case 2}. $v$ is non-trivial and $w$ is trivial,
hence $N=\Nw$. Then we can view $v$ as a prime 
divisor of $M\md N$, and in the above notation 
one has: Let $\clT\subset M$ be a preimage of a 
separable transcendence basis $\clT_0\subset\Mv$ 
of $\Mv\md N$, and $N_\clT=N(\clT)$. Then 
$w_\clT:=v|_{N_\clT}$ is trivial, and the 
relative algebraic closure $M^0$ of $N(\clT)$ 
in $M^h$ is a field of representatives for $\Mv$. 
In particular, if $\pi\in M$ has $v(\pi)=1$, then 
$M^h=M^0(\pi)^h$ by Fact~\ref{fkt8}. 
\vskip2pt
Next, let $\Us\clT\subset\Us k$ be a preimage
of $\imath(\clT_0)\subset\Uj\kappa$ under the
canonical residue map $\Us\clO\to\Uj\kappa$.
Then $\Us v$ is trivial on $N_{\Us\clT}=\Us N(\Us\clT)$,
and $\kappa_M=\imath(\Mv)$ has a unique preimage
$\Us M^0\subset \Us k$ which is algebraic over
$N_{\Us\clT}$. Finally, the $k$-isomorphism 
$M_0\to\Mv\to\Us M^0$ together with $\pi\mapsto\Uj\pi$ 
give rise to a $k$-embeddings of fields 
\[
\Us\imath:M\hra M^h=M^0(\pi)^h\hor{\cong}
\Us M^0(\Us\pi)^h\subset\Us k,
\] 
whose restriction to $N=\Nw$ is $\Us\jmath$. Hence 
the $\Us k$-rational point $\tilde x\in X(\Us k)$ 
defined by $\Us\imath:M\hra\Us k$ satisfies 
$f^{\Us k}(\tilde x)=\tilde y$.
\vskip5pt
\underbar{Case 3}. $w$ is non-trivial. Let 
$\pi\in N$ be such that $w(\pi)=1$, hence 
$v(\pi)=1$ by the fact that $e(v|w)=1$. Then 
$N_\clT=N(\clT)\hra M$ gives rise to the 
embedding of the Henselizations
$M^h\md N_\clT^h$. Reasoning as above, the 
unique unramified subextension $M_0\md N_\clT^h$
of $M^h\md N_\clT^h$ satisfies $M^h=M_0$, and
$N_\clT\to N_{\Us\clT}$ together with $\pi\mapsto\Us\pi$,
gives rise to a $k$-embedding $\Us\imath:M\to\Us k$
prolonging $\Us\jmath:N\to k$, etc. Hence finally,
one gets a point $\tilde x\in X(\Us k)$ such that
$f^{\Us k}(\tilde x)=\tilde y$.
\vskip5pt
\noindent
\subsubsection{Proof of assertion {\rm 2)} 
of the Key Lemma \ref{thekeylmm}} $\ha0$
\vskip2pt
Since the implication $\Rightarrow$ is actually
assertion~1) of the Key lemma, it is left to
prove the converse implication, that is, that
property \Srj$_{\Sgm k}$ implies the hypothesis
\Bstar$_{\Sgm k}$.
In Notations/Remarks~\ref{notarem2}, 
suppose that $f^{\Us k}:X(\Us k)\to Y(\Us k)$ 
is surjective for a given $\clU$. Let 
$y\in Y$, $N:=L_y=\kappa(y)$, 
and $w:=w_y\in\clD(N|k)$, and a 
$k$-embedding $\jmath:\Nw=l_y\hra\Uj\kappa$
be given. We show that there is $x\in X_y$
such that setting $M:=\kappa(x)$ there is
$v\in\clD(M|k)$ such that $w=v|_N$, 
$e(v|w)=1$, and $\Nw=l_y\hra k_x=\Mv$ is \jspl.
\vskip2pt
Indeed, given $w$, we define a particular 
$\Us k$-rational point $\tilde y=\tilde y_w\in Y(\Us k)$
as follows: First, if $w$ is trivial, let 
$\tilde y_w$ be defined by the $k$-embedding
$\jmath:N=\kappa(y)\hra\Uj\kappa\subset\Us k$.
Second, if $w$ is non-trivial, hence a prime 
divisor of $N|k$, let $\kappa_w\subset N^h$
be a field of representatives for $\Nw$. (Note
that since $\chr(k)=0$, such a field of representatives 
exists.) Thus by Fact~\ref{fkt8}, one has
$N^h=\kappa_w(\pi)^h\!$. Hence setting 
$\kappa'_w=\jmath(\Nw)\subset\Uj\kappa\subset\Us k$,
one has that $N^h$ has a canonical $k$-embedding 
$\Us\jmath^h:N^h=\kappa_w(\pi)^h\to
\kappa'_w(\Us\pi)^h\subset\Us k$ via 
$\jmath:\kappa_w\to\Nw\to\kappa'_w\subset\Uj\kappa$, 
$\pi\mapsto\Uj\pi$. 
\vskip2pt
Let $\tilde y\in Y(\Us k)$ be defined by the 
$k$-embedding $\Us\jmath:=\Us\jmath^h|_N:
N\hra N^h\hra\Us k$. Then
by property \Srj$_{\Sgm k}$, there is some
$\tilde x\in X(\Us k)$ such that
$f^{\Us k}(\tilde x)=\tilde y$, and let $\tilde x$
be defined by some $x\in X$ and a $k$-embedding
$\Us\imath:M=\kappa(x)\hra\Us k$.
Then by mere definition one has $f(x)=y$, and
the canonical $k$-embedding 
$f_{xy}:N=\kappa(y)\hra\kappa(x)=M$ satisfies
$\Us\imath\circ f_{xy}=\Us\jmath$. Hence
setting $v:=\Us v|_M$, one has $w=v|_N$,
and the following hold: First, one has a canonical 
$k$-embedding $\Nw\hra\Mv\hra\Uj\kappa$.
Second, one has canonical embeddings 
$wN\hra vM\hra\Us v\Us k$; and if
$w$ is non-trivial, then by the definition 
of $w$ one has: $w(\pi)=1=\Us v(\Uj\pi)$,
hence $wN\hra vM\hra\Us v\Us k$ are 
isomorphisms, and $e(v|w)=1$. Finally, since 
$\jmath:\Nw\hra\Uj\kappa$ prolongs to a 
$k$-embedding $\Mv\hra\Uj\kappa$, it follows 
that $\Mv|\Nw$ is \jspl.
%
%
\subsection{Final Remarks} $\ha0$
\vskip2pt
\noindent
First, it is believed that the hypothesis 
\defi{(qAKE)}$_{\Sgm k}$ always holds, in 
particular, assertion~1) of Theorem~\ref{thm2} 
should hold unconditionally. Second, the 
question whether assertion~2) of 
Theorem~\ref{thm2} holds in positive 
characteristic, is related to subtle questions 
concerning the relationship between 
ramification index and purely inseparable 
non-liftable extensions of the residue field of 
prime divisors. Hence it is an interesting (and 
maybe subtle) question whether assertion~2) 
of Theorem~\ref{thm2} holds (un)conditionally 
in positive characteristic.  
\section{Proof of Theorem~\ref{thm0}}
By mere definitions, Theorem~\ref{thm0} 
is a consequence of Theorem~\ref{thm1} 
and Theorem~\ref{thm3} below. The latter 
relates the pseudo-splitness of prime divisors in 
extensions of function fields and pseudo-splitness 
of morphisms of proper integral v.r.l.\ varieties
over arbitrary fields $k$. We say that a function 
field $N|k$ is \defi{valuation-regular-like}, 
if $N|k$ has a co-final system of {\it proper 
\vrl\/}\ha3models~$(Z_\alpha)_\alpha$. 
By Hironaka's Desingularization Theorem, one has:
\vskip5pt
\centerline{\it If $\,\chr(k)=0$, every function 
field $N\md k$ is v.r.l.\/}
\begin{theorem}
\label{thm3}
Let $f:X\to Y$ be a dominant morphism of proper
\vrl\ $k$-varieties, and suppose that $K=k(X)$, 
$L=k(Y)$ are \vrl.
Then $f:X\to Y$ is pseudo-split in the sense 
defined in the Introduction iff $\,\clD(L|k)$ 
is pseudo-split in $\clD(K|k)$.
\end{theorem}
\begin{proof}
Since $K\md L$ is an extension of \vrl\
function fields over $k$, there are cofinal systems 
$(f_\alpha:X_\alpha\to Y_\alpha)_{\alpha\in I}$
of dominant morphisms of proper \vrl\ 
$k$-varieties defining $K\md L$. In particular, the 
structure morphisms $X_{\alpha''}\to X_{\alpha'}$
and $Y_{\alpha''}\to Y_{\alpha'}$, 
$\alpha'\leqslant\alpha''$ are proper. 
Further, w.l.o.g., we can and will replace the 
given projective system $(f_\alpha)_{\alpha\in I}$ 
by any subsytem $(f_{\alpha'})_{\alpha'\in I'}$ 
indexed by any co-final segment $I'\subset I$. 
In particular, w.l.o.g., we can and \underbar{will} 
suppose that every $f_\alpha:X_\alpha\to Y_\alpha$ 
dominates the given $f:X\to Y\!$. 
\vskip2pt
We conclude this preparation by summarizing a few well 
known facts, to be used later.
%
%
\begin{fact}
\label{fkt5}
{\it Let $v\in\Val k N$ have center 
$x_\alpha\in X_\alpha$ for $\alpha\in I$. 
Setting $w:=v|_L$, one has:
\vskip2pt
\itm{35}{
\item[{\rm1)}] The center of $w$ on $Y_\alpha$
is $y_\alpha=f_\alpha(x_\alpha)$, and one has: 
\[
\eum_v=\cup_\alpha\eum_{x_\alpha}
\subset\,\cup_\alpha\clO_{x_\alpha}\!=\clO_v,
\quad \eum_w=\cup_\alpha\eum_{y_\alpha}
\subset\,\cup_\alpha\clO_{y_\alpha}\!=\clO_w,
\]
and therefore, 
$\Lw=\cup_\alpha\kappa(y_\alpha)
\hra\cup_\alpha\kappa(x_\alpha)=\Kv$
canonically.
\vskip3pt
\item[{\rm2)}] If $v\in\clD(K|k)$, then
$\clO_v=\clO_{x_\alpha}$ and 
$\clO_w=\clO_{y_\alpha}$ for 
$\alpha\in I_v$ in a cofinal segment 
$I_v\subset I$.  
}
}
\end{fact}
\subsection{The implication ``$\,\Rightarrow\,$''}$\ha0$
\vskip5pt
Given $w\in\clD(L|k)$, we show that $w$ is
pseudo-split in $\clD(K|k)$.
\vskip5pt
\underbar{Case 1}. $w$ is the trivial valuation
of $L|k$. Then the center $y\in Y$ of $w$ is the 
generic point $y=\eta_Y$ of $Y\!$, and $X_y=X_L$
is the generic fiber of $f:X\to Y\!$. Further,
$\Lw=L$. Since $w$ is pseudo-split under~$f$, 
for every co-procyclic extension $l'\md L$ there 
exists $x\in X_L$ and $v_x\in\clD(K_x|k)$ 
with $k_x\md L$ split above $l'$. Since $K|k$ 
is valuation-regular-like, there exists $\tlv\in\Val k K$ 
having center $x\in X_L\subset X$ such that 
$K\hb1\tlv=K_x$. The valuation theoretical 
composition $v:=v_x\circ\tlv$ is trivial on 
$L$ under $L\hra K$, hence $w=v|_L$, and 
$k_x=K_xv_x=\Kv$. Further, by Fact~\ref{fkt5}, 
if $x_\alpha\in X_\alpha$ is the center of $v$ 
on $X_\alpha$, one has 
$\clO_v=\cup_\alpha\clO_{x_\alpha}$,  
$\eum_v=\cup_\alpha\eum_{x_\alpha}$, and 
$k_x=\Kv=\cup_\alpha\kappa(x_\alpha)$.
In particular, since $k_x\md k$ is finitely 
generated, there exists a cofinal segment
$I_x\subset I$ such that $k_x=\Kv=\kappa(x_\alpha)$
for all $\alpha\in I_x$. Since $K|k$ is 
valuation-regular-like, for every $x_\alpha\in X_\alpha$, 
there exists $v_\alpha\in\clD(K|k)$ with 
center $x_\alpha\in X_\alpha$ such that 
$\Kv_\alpha\md\kappa(x_\alpha)$ is a regular 
field extension. In particular, for $\alpha\in I_x$
one has: $k_x=\kappa(x_\alpha)$ and 
$k_x=\kappa(x_\alpha)\hra\Kv_\alpha$ is a
regular field extension. Hence since $k_x\md L$
is split above $l'\!$, and $\Kv_\alpha\md k_x$
is a regular extension, by transitivity of splitness,
it follows that $\Kv_\alpha\md L$
is split above $l'\!$. Finally, since $v_\alpha|_L$ 
is trivial, hence $w=v_\alpha|_L$, it follows that 
$w$ is pseudo-split in $\clD(K|k)$, as claimed.
\vskip5pt
\underbar{Case 2}. $w$ is non-trivial, hence 
$w\in\clD(L|k)$ is a prime divisor prime 
divisor of $L|k$. Let $y_\alpha\in Y_\alpha$ be 
the center of $w$ on $Y_\alpha$. By Fact~\ref{fkt5}, 
there is a  co-final segment $I_w\subset I$ such that 
$\clO_w=\clO_{y_\alpha}$, thus $\eum_w=
\eum_{y_\alpha}$ and $\Lw=\kappa(y_\alpha)$
for $\alpha\in I_w$. Letting $y=\eta_Y$ be 
the generic point of $Y$, one has $L_y=L$, 
and $w\in\clD(L_y)$, and $X_y=X_L$ is the 
generic fiber of $f:X\to Y\!$.
\vskip2pt
Let $l'\md\Lw$ be a co-procyclic extension. Then
$w\in\clD(L_y|k)$ being split under $f$ implies 
that there is $x\in X_y=X_L$ and a prime 
divisor $v_x\in\clD(K_x|k)$ with $w=v_x|_L$ 
under $L\hra K_x$ such that $e(v_x|w)=1$ and 
$k_x\md\Lw$ is split above $l'\!$. Let $\pi\in L$ 
satisfy $w(\pi)=1$, hence in particular, $v_x(\pi)=1$ 
under the $k$-embedding $L=L_y\hra K_x$. 
Since $K|k$ is valuation-regular-like, there is 
$\tlv\in\Val I K$ with center $x\in X$ and 
$K\tilde v=\kappa(x)=K_x$. In particular, 
$\tlv|_L$ is trivial on $L$ under $L\hra K$,
and the valuation theoretical composition 
$v:=v_x\circ\tlv\in\Val k K$ satisfies:
\vskip2pt
\itm{25}{
\item[a)] $\Kv=K_xv_x=k_x$, and $w=v|_L$ 
under $L\hra K$, thus $\clO_w=\clO_v\cap L$.
\vskip2pt
\item[b)] Since $wL=v_xK_x\hra vK$, it follows that 
$v(\pi)$ is the minimal positive element of $vK$.
\vskip2pt
\item[c)] In particular, $\eum_v=\pi\clO_v$, hence 
$\pi\in\eum_v\backslash\ha1\eum_v^2$.
}
Recalling that $f_\alpha:X_\alpha\to Y_\alpha$
are proper morphisms, since $w=v|_L$ has
the center $y_\alpha\in Y_\alpha$, it follows that
$v$ has a (unique) center $x_\alpha\in X_\alpha$, 
and $f(x_\alpha)=y_\alpha$. In particular, since 
$w=v|_L$, by Fact~\ref{fkt5} one has: First, 
since $k_x=\Kv$ is finitely generated over $k$, 
there is a cofinal segment $I_x\subset I$ such 
that $\Kv=k_x=\kappa(x_\alpha)$ for all 
$\alpha\in I_x$. Recalling that $I_w\subset I$ is a 
cofinal segment such that $\Lw=\kappa(y_\alpha)$ 
for all $\alpha\in I_w$, it follows that $I':=I_w\cap I_x$ 
is a cofinal segment in $I$ such that for all $\alpha\in I'$ 
the following hold: 
\[
\clO_w=\clO_{y_\alpha}=\clO_{x_\alpha}\cap L,
\quad\eum_w=\eum_{y_\alpha}=\eum_{x_\alpha}\cap L,
\quad\Lw=\kappa(y_\alpha)\hra\kappa(x_\alpha)=k_x.
\]
In particular, $\pi\in\eum_{x_\alpha}$, and since 
$\pi\not\in\eum_v^2$, one has that $\pi\not\in
\eum_{x_\alpha}^2$ for all $\alpha\in I'\!$.
\vskip2pt
Since $K|k$ is valuation regular-like, there exists 
$v_\alpha\in\clD(K|k)$ with center 
$x_\alpha\in X_\alpha$ such that 
$\Kv_\alpha\md\kappa(x_\alpha)$ is a regular
field extension, 
and $v_\alpha(\pi)=1$, because $\pi\in\eum_{x_\alpha}
\backslash\eum_{x_\alpha}^2$. Therefore
$w_\alpha:=v_\alpha|_L$ lies in $\clD(L|k)$, and 
$\eum_{w_\alpha}=\eum_{v_\alpha}\cap L$. 
Since  $\eum_{y_\alpha}=
\eum_{x_\alpha}\cap L$, one has:
\[
\eum_w=\eum_{y_\alpha}=\eum_{x_\alpha}\cap L
\ \subset \ \eum_{v_\alpha}\cap L=\eum_{w_\alpha},
\ \ \hbox{ hence } \ \ \clO_w\supset\clO_{w_\alpha}.
\]
Since $w$, $w_\alpha$ are discrete valuations, 
one must have $w=w_\alpha$. Recalling that 
$w_\alpha:=v_\alpha|_L$, we finally get 
$w=v_\alpha|_L$, hence $e(v_\alpha|w)=1$, because 
$v_\alpha(\pi)=1=w_\alpha(\pi)$. Since $k_x\md\Lw$ 
is $l'$-split and $k_x=\kappa(x_\alpha)\hra\Kv_\alpha$ 
is a regular field extension for $\alpha\in I'\!$, it 
follows that $\Kv_\alpha\md \Lw$ is split above $l'$. 
Conclude that $w$ is pseudo-split in $\clD(K|k)$.
\subsection{The implication ``$\,\Leftarrow\,$''}$\ha0$
\vskip5pt
Setting $L_y:=\kappa(y)$ for $y\in Y\!$, we have to 
show that every $w_y\in\clD(L_y|k)$ is pseudo-split 
under~$f$ in the sense defined in the Introduction. 
First, if $y=\eta_Y$ is the generic point of $Y\!$, hence 
$L_y=k(Y)=:L$, then the implication follows directly 
from the fact that $\clD(L|k)$ is pseudo-split in 
$\clD(K|k)$. Hence w.l.o.g., $y\neq\eta_Y$.
\vskip5pt
\underbar{Case 1}. $w_y$ is the trivial valuation 
of $L_y$, i.e., $L_y=\Lw_y=l_y$. First, since 
$L|k$ is regular-like, there exists $w\in\clD(L|k)$ 
having center $y\in Y$ such that $\Lw\md l_y$
is a regular field extension. Let $l'\md l_y$ be a 
co-procyclic extension, and $l'_w\md\Lw$ be a
co-procyclic extension with $l'=\oli{\,l\,}\!_y\cap l'_w$.
Since $\clD(L|k)$ is pseudo-split in $\clD(K|k)$, 
there is $v\in\clD(K|k)$ such that $e(v|w)=1$ and 
$\Kv\md\Lw$ is split above $l'_w$. Hence if 
$x\in X$ is the center of $v$ on $X$, then 
$y=f(x)$ is the center of $w$ on $Y\!$,   
$x$ lies in the fiber $x\in X_y$ of $f$ at $y$, 
and there are canonical $k$-embeddings
\[
L_y\hra K_x=\kappa(x)\hra\Kv\,.
\]
In particular, since $\Kv\md\Lw$ is split above 
$l'_w$ and $l'=\oli{\,l\,}\!_y\cap l'_w$, and
$\Lw\md L_y$ is a regular field extension, by 
the transitivity of splitness, it follows that 
$\Kv\md L_y$ is split above $l'$. Hence finally,
since $K_x\md L_y$ is a subextension of $\Kv\md L_y$, 
it follows that $K_x\md L_y$ is split above $l'$. 
Hence letting $v_x$ be the trivial valuation of 
$K_x$, it follows that $e(v_x|w_y)=1$, and 
$l_y=L_y\hra K_x=k_x$ is split above $l'\!$,
as claimed.
\vskip5pt
\underbar{Case 2}. $w_y\in\clD(L_y|k)$ is 
non-trivial. Since $L|k$ is valuation-regular-like, 
there exists $w_L\in\clD(L|k)$ having center 
$y$ on $Y$ and $\Lw_L=L_y$. Next let 
$w:=w_y\circ w_L$ be the valuation 
theoretical composition of $w_y$ and $w_L$, 
hence $Lw=L_yw_y=l_y$, and 
$wL=w_yL_y\times w_LL$ lexicographically 
ordered. In particular, if $\pi\in\clO_{w_L}$ 
is any element whose image in $L_y$ is a  
uniformizing parameter of $w_y$, then 
$1_y=w(\pi)\in wL$ is the unique minimal 
positive element, and $\eum_w=\pi\clO_w$. 
Then letting $y_\alpha\in Y_\alpha$ 
be the center of $w$ on $Y_\alpha$, one has: 
$\clO_w=\cup_\alpha\clO_{y_\alpha}$,
$\eum_w=\cup_\alpha\eum_{y_\alpha}$ and 
$\eum_{y_\alpha}=\eum_w\cap\clO_{y_\alpha}$, 
thus $l_y=\Lw=\cup_\alpha\kappa(y_\alpha)$.
In particular, since $\pi\in\eum_w$ and 
$l_y\md k$ is finitely generated, there is a cofinal 
segment $I_y\subset I$ such that the following hold: 
\vskip2pt
\itm{25}{
\item[a)] $\pi\not\in\eum_{y_\alpha}$ for all 
$\alpha\in I$, and $\pi\in\eum_{y_\alpha}$ 
for all $\alpha\in I_y$.
\vskip2pt
\item[b)] $\kappa(y_\alpha)\subset l_y$ for 
all $\alpha\in I$, and $\kappa(y_\alpha)=l_y$ 
for all $\alpha\in I_y$. 
}
Finally, since $L|k$ is valuation-regular-like, 
taking into account Fact~\ref{fkt5}, there is 
$w_\alpha\in\clD(L|k)$ such that 
$\clO_{w_\alpha}$ dominates the local ring 
$\clO_{y_\alpha}$, and further: $w_\alpha(a)=1$ 
for all $a\in\eum_{y_\alpha}\hb3\backslash
\eum_{y_\alpha}^2$, and $\kappa(y_\alpha)\hra
\Lw_\alpha:=l_\alpha$ is a regular field extension.  
Therefore, for $\alpha\in I_y$ the following hold:
\vskip2pt
\itm{20}{
\item[-] $l_y=\kappa(y_\alpha)\hra l_\alpha$ 
is a regular field extension. 
\vskip2pt
\item[-] $w_\alpha(\pi)=1$, hence $\pi$ generates 
$\eum_{w_\alpha}$.
}
Next let $l'\md l_y$ be a co-procyclic extension,
and using that $l_\alpha\md l_y$ is a regular 
field extension, let $l'_\alpha\md l_\alpha$ 
be any co-procyclic extension such that 
$l'=\oli {\ha1l\ha1}\!_y\cap l'_\alpha$. Since 
$\clD(L|k)$ is pseudo-split in $\clD(K|k)$, 
there is~a~prime divisor $v_\alpha\in\clD(K|k)$ 
with $w_\alpha\!=v_\alpha|_L$ such that 
$e(v_\alpha|w_\alpha)=1$,  and setting 
$k_\alpha:=\Kv_\alpha$ one has: $k_\alpha\md l_\alpha$
is $l'_\alpha$-split. Then taking into account 
the {\it transitivity of splitness,\/} since
$l'=\oli {\ha1l\,}\!_y\cap (l'l_\alpha)=
\oli {\ha1l\ha1}\!_y\cap l_\alpha$, one 
finally gets:
\vskip2pt
\itm{20}{
\item[$\bullet$] {\it $w_\alpha(\pi)=1=v_\alpha(\pi)$ 
under $L\hra K$, and $k_\alpha\md l_\alpha$ 
is split above $l'l_\alpha$ for $\alpha\in I_y$.\/}
}
Now let $\clP_I$ be the pre-filter on $I$ 
formed by the cofinite subsets $I'\subset I_y$, 
and $\,\,\clU$ be an ultrafilter on $I$ containing 
$\clP_I$. Consider the corresponding ultrapowers 
$\Ust L\hra \Ust K$ of $L\hra K$, endowed with 
the corresponding ultraproducts of valuations rings
\[
\clO_{\Ust w}={\textstyle\prod}_\alpha
\clO_{w_\alpha}/\,\clU\hra{\textstyle\prod}_\alpha
\clO_{v_\alpha}/\,\clU=:\clO_{\Ust v},
\]
having value groups and residue fields as follows:
\[
\Ust w\Ust L=\lvZ^I/\,\clU=\Ust v\Ust K,\quad
\Ust l={\textstyle\prod}_\alpha l_\alpha/\,\clU=\Ust L\Ust w
\hra\Ust K\Ust v={\textstyle\prod}_\alpha k_\alpha/\,\clU=\Ust k.
\]
One has: First, since for every $\alpha'\!,\alpha''\in I$ 
there exists $\alpha\in I_y$ with $\clO_{y_{\alpha'}},
\clO_{y_{\alpha''}}\subset\clO_{y_\alpha}
\subset\clO_{w_\alpha}$, it follows that 
$\clO_w=\cup_\alpha\clO_{y_\alpha}\subset\clO_{\Ust w}$. 
Hence setting $w':=\Ust w|_L$, 
$v:=\Ust v|_K$, one finally has: First, $\clO_w\subset
\clO_{w'}\subset\clO_v$, where the later inclusion is 
defined via $L\hra K$. Second, since $w_\alpha(\pi)=1=v_\alpha(\pi)$,
it follows that $\Ust w(\pi)=\Ust1=\Ust v(\pi)$ is the 
minimal positive element in both 
value groups $\Ust w \Ust L=\Ust v \Ust K$.
Therefore, $\pi\in\clO_w,\clO_{w'}$ is the element of
minimal positive value, thus $\clO_w=\clO_{w'}$ by 
general valuation theory. Further, since $\pi\in\clO_{\Ust v}$
is an element of minimal positive value, it follows that
$\pi\in\clO_v$ is an element of minimal positive value
as well. Finally, recalling that $w=w_y\circ w_L$, 
hence by mere definitions, $\clO_{w_L}=\clO_w[1/\pi]$,
it follows that $\clO_{v_K}:=\clO_v[1/\pi]$ is 
a $k$-valuation ring of $K$ such that $\clO_{v_K}
\cap L=\clO_{w_L}$. In particular, since $X$ is
proper, $v_K$ has a center on $X$, say $x\in X$. 
One has:
\vskip2pt
\itm{25}{
\item[a)] Since $\clO_{w_L}\hra\clO_{v_K}$ under
$L\hra K$, one has $f(x)=y$, thus $L_y=\kappa(y)
\hra\kappa(x)=:K_x$.
\vskip2pt
\item[b)] $\clO_{w_y}\!=\clO_w/\eum_{w_L}\!\hra
\clO_v/\eum_{v_K}\!=:\!\clO_{v_x}$ are DVRs of 
$L_y\hra K_x$ with $w_y(\pi)=1=v_x(\pi)$, 
thus $\eum_{w_y}=\pi\clO_{w_y}\subset
        \pi\clO_{v_x}=\eum_{v_x}$, and 
$k_x\md l_y$ is a $k$-subextension of 
$\Ust k\md\Ust l$. 
\vskip2pt
\item[c)] Since $k_\alpha\md l_\alpha$ is split 
above $l'l_\alpha$ for all $\alpha\in I_y$, it 
follows that $\Ust k\md\Ust l$ is split above $l'\Ust l$, 
thus $k_x\md l_y$ is split above $l'\!$ by the
transitivity of splitness. 
}
Hence $e(v_x|w_y)=1$, and $k_x\md l_y$ is split 
above $l'\!$, thus completing the proof of Case~2.
\vskip7pt
\noindent
This completes the proof of Theorem~\ref{thm3}.
\end{proof} 
%
%
%
%
%
%

\end{document}